\documentclass[a4paper]{amsart}

\usepackage{amscd,amsthm,amsfonts,amssymb,amsmath,esint}
\usepackage[colorlinks=true]{hyperref}
\usepackage[all]{xy}
\usepackage[OT2,T1]{fontenc}

\usepackage{mathrsfs}

\usepackage{multirow}

\newcommand{\msa}{\mathscr{A}}\newcommand{\msb}{\mathscr{B}}
\newcommand{\mse}{\mathscr{E}}

    \newcommand{\BC}{{\mathbb {C}}} 
     \newcommand{\BF}{{\mathbb {F}}}

     \newcommand{\BN}{{\mathbb {N}}}
     
    \newcommand{\BQ}{{\mathbb {Q}}} \newcommand{\BR}{{\mathbb {R}}}

     \newcommand{\BZ}{{\mathbb {Z}}}

    \newcommand{\CA}{{\mathcal {A}}} 
    \newcommand{\CC}{{\mathcal {C}}} 
    \newcommand{\CE}{{\mathcal {E}}} 
     \newcommand{\CH}{{\mathcal {H}}}
     
    \newcommand{\CK}{{\mathcal {K}}} \newcommand{\CL}{{\mathcal {L}}}
     
    \newcommand{\CO}{{\mathcal {O}}}

     \newcommand{\fd}{{\mathfrak{d}}}

    \newcommand{\fm}{{\mathfrak{m}}} 
     \newcommand{\fp}{{\mathfrak{p}}}
    \newcommand{\fq}{{\mathfrak{q}}}

    \newcommand{\fM}{{\mathfrak{M}}} 
    \newcommand{\fO}{{\mathfrak{O}}}

    \newcommand{\Gal}{{\mathrm{Gal}}} 
    
    \newcommand{\Hom}{{\mathrm{Hom}}}
    
    \renewcommand{\Im}{{\mathrm{Im}}}

    \newcommand{\ord}{{\mathrm{ord}}} \newcommand{\rank}{{\mathrm{rank}}}

    \renewcommand{\mod}{\ \mathrm{mod}\ }\renewcommand{\Re}{{\mathrm{Re}}}

    \newcommand{\Sel}{{\mathrm{Sel}}}

\DeclareSymbolFont{cyrletters}{OT2}{wncyr}{m}{n}
\DeclareMathSymbol{\Sha}{\mathalpha}{cyrletters}{"58}
    \newcommand{\wh}{\widehat}

    \newcommand{\ov}{\overline}

    \newcommand{\ra}{\rightarrow} 
    \newcommand{\bs}{\backslash}
    \newcommand{\nequiv}{\equiv\hspace{-10pt}/\ }

    \theoremstyle{plain}
    \newtheorem{thm}{Theorem}[section] \newtheorem{cor}[thm]{Corollary}
    \newtheorem{lem}[thm]{Lemma}  \newtheorem{prop}[thm]{Proposition}

\theoremstyle{remark} 
\theoremstyle{remark} 
\theoremstyle{remark} 

    \newcommand{\Neron}{N\'{e}ron~}

    \numberwithin{equation}{section}

\begin{document}

\title[Elliptic Curves with Rank One and Non-Trivial {$\Sha[2]$}]{A classical family of elliptic curves having rank one and the $2$-primary part of their Tate--Shafarevich group non-trivial}

\author{Yukako Kezuka,\quad Yongxiong Li}

\thanks{The first author was supported by the DFG grant: SFB 1085 ``Higher invariants''. The second author is supported by NSFC-11901332}

\begin{abstract}
We study elliptic curves of the form $x^3+y^3=2p$ and  $x^3+y^3=2p^2$ where $p$ is any odd prime satisfying $p\equiv 2\bmod 9$ or $p\equiv 5\bmod 9$. We first show that the $3$-part of the Birch--Swinnerton-Dyer conjecture holds for these curves. Then we relate their $2$-Selmer group to the $2$-rank of the ideal class group of $\BQ(\sqrt[3]{p})$ to obtain some examples of elliptic curves with rank one and non-trivial $2$-part of the Tate--Shafarevich group.
\end{abstract}

\maketitle

\section{Introduction and the main results}

Let $n\geq 3$ be a cube free integer. The elliptic curve
\[C_n\colon x^3+y^3=n\]
is the twist of $C\colon x^3+y^3=1$ by the cubic field $\BQ(\sqrt[3]{n})$, where $\sqrt[3]{n}$ denotes the real root. The study of the arithmetic of $C_n$ is very old. In particular, it is well known that $C_n$ has no non-zero rational torsion, and thus $C_n(\BQ)$ is infinite precisely when $n$ is the sum of two rational cubes.
Let $p$ be any odd prime satisfying
\begin{equation}\label{bcong}
 p \equiv 2 \mod 9 \, \, {\rm or} \, \,  p \equiv 5 \mod 9.
\end{equation}
In \cite{satge}, Satg\'{e} used the theory of Heegner points to show that $C_ {2p}$ has rank 1 when $p\equiv 2 \mod 9$, and $C_{2p^2}$ has rank 1 when $p \equiv 5 \mod 9$. More recently, it was shown in  \cite{CST} that, for all odd primes $p$ satisfying \eqref{bcong} the $3$-part of the Birch--Swinnerton-Dyer conjecture holds for the product of the two curves $C_{2p}\times C_{2p^2}$;  here, we remark that one of the curves in this product is of rank zero, while the other is of rank one.

The first main result of this paper is the proof of the $3$-part of the Birch--Swinnerton-Dyer conjecture for each of the individual  curves $C_{2p}$ and $C_{2p^2}$ for all odd primes $p$ satisfying \eqref{bcong}.  More precisely, we establish the following results. Let $L(C_n, s)$ be the complex $L$-series of $C_n$ over $\BQ$. Since $C_n$ has complex multiplication by the integer ring $\CO_K$ of $K=\BQ(\sqrt{-3})$, the analytic continuation and functional equation of $L(C_n, s)$ are known by Deuring's theorem. Let~$A$ be the elliptic curve with classical Weierstrass equation
\[A\colon y^2=4x^3-1.\]
Then fixing an embedding of $K $ into $\BC$, and noting that $A$ also has complex multiplication by $\CO_K$, the lattice $L$ of complex periods of the differential $dx/y$ on $A$ is of the form $\Omega\CO_K$, with $\Omega = 3.059908...$, a real number. For each integer $n \geq 3$, we then define
\begin{equation}\label{per}
\Omega_n = \Omega/(\sqrt{3}n^{1/3}).
\end{equation}
The rationality of the algebraic part $\frac{L(C_n,1)}{\Omega_n}$ of ${L(C_n,s)}$ at $s=1$ follows from a result of Birch using modular symbols, which deals with any elliptic curve over~$\BQ$. Moreover, Stephens strengthened this result for $C_n$ in \cite{stephens} to show that this is a rational integer.

\begin{thm}\label{main1} Let $p$ be an odd prime satisfying \eqref{bcong}. Then
\begin{enumerate}
  \item If $p\equiv 5\mod 9$, the $L$-values $L(C_{2p},1)$ and $L(C_{2^2p^2},1)$ are both non-zero, and the $3$-part of the Birch--Swinnerton-Dyer conjecture holds for $C_{2p}$ and $C_{2^2p^2}$.  Moreover, we have the following congruence
      \[\frac{L(C_{2p},1)}{3\Omega_{2p}}\equiv \frac{L(C_{2^2p^2},1)}{3\Omega_{2^2p^2}}\equiv 1\mod 3.\]
  \item If $p\equiv 2\mod 9$, the $L$-values $L(C_{2p^2},1)$ and $L(C_{2^2p},1)$ are both non-zero, and the $3$-part of the Birch--Swinnerton-Dyer conjecture holds for $C_{2p^2}$ and $C_{2^2p}$. Moreover, we have the following congruence
      \[\frac{L(C_{2^2p},1)}{3\Omega_{2^2p}}\equiv \frac{L(C_{2p^2},1)}{3\Omega_{2p^2}}\equiv 1\mod 3.\]
\end{enumerate}
\end{thm}

We write $\Sha(C_n)$ for the Tate--Shafarevich group of $C_n$ over $\BQ$. Combined with the results in \cite{CST}, we obtain the following corollary of Theorem \ref{main1}.

\begin{cor}\label{main1+}
Let $p$ be an odd prime satisfying \eqref{bcong}.
\begin{enumerate}
  \item If $p\equiv 2\mod 9$, the curve $C_{2p}$ has rank $1$ and $L(C_{2p},s)$ has a simple zero at $s=1$. Furthermore, $\Sha(C_{2p})$ is finite and the $3$-part of the Birch--Swinnerton-Dyer conjecture holds for $C_{2p}$.
  \item If $p\equiv 5\mod 9$, the curve $C_{2p^2}$ has rank $1$ and $L(C_{2p^2},s)$ has a simple zero at $s=1$. Furthermore, $\Sha(C_{2p^2})$ is finite and the $3$-part of the Birch--Swinnerton-Dyer conjecture holds for $C_{2p^2}$.
\end{enumerate}

\end{cor}

\bigskip

The second main result of this paper is on the $2$-part of the Tate--Shafarevich group of the curves $C_{2p}$ and $C_{2p^2}$ in the rank $1$ case studied in Corollary \ref{main1+}. For any abelian group (scheme) $\fM$ (over a base scheme $S$) and a positive integer $m$, we write $\fM[m]$ for the group (scheme) of $m$-torsion of $\fM$. We relate the $2$-Selmer group of these curves to the $2$-part of the ideal class group of the field $L=\BQ(\sqrt[3]{p})$. We remark that such a relation has already been obtained in \cite{cl} for a wide class of elliptic curves without complex multiplication. We can apply the same methods to our curves on noting that their Mordell--Weil groups have trivial torsion subgroups and that for any prime number $\ell$, we have $\bold{\Phi}[2]=0$, where $\bold{\Phi}/\overline{\mathbb{F}}_\ell$ denotes the component group scheme of a corresponding \Neron  model over $\BQ_\ell$. In addition, in our case, we can use this relation to obtain examples of curves with rank $1$ and non-trivial $2$-part of the Tate--Shafarevich group.

Given a number field $F$, let $Cl(F)$ be the ideal class group of $F$. We call the integer $\rank_2\left(Cl(F)\right):=\dim_{\BF_2}(Cl(F)/2 Cl(F))$ the $2$-rank of $Cl(F)$. Here, $\BF_2$ denotes the finite field of two elements. 

\begin{thm}\label{main2}
Let $p$ be an odd prime satisfying \eqref{bcong}.
\begin{enumerate}
\item If $p\equiv 2\mod 9$, we have $\rank_2(Cl(L))\geq 2$ if and only if $\Sha(C_{2p})[2]$ is non-trivial. In particular, in the case $\rank_2(Cl(L))\geq 2$, the curve $C_{2p}$ has rank $1$ and has non-trivial $\Sha(C_{2p})[2]$.

\item If $p\equiv 5\mod 9$, we have $\rank_2(Cl(L))\geq 2$ if and only if $\Sha(C_{2p^2})[2]$ is non-trivial. In particular, in the case $\rank_2(Cl(L))\geq 2$, the curve $C_{2p^2}$ has rank $1$ and has non-trivial $\Sha(C_{2p^2})[2]$.
\end{enumerate}

\end{thm}

We remark that, although we cannot show at present that there are infinitely many primes $p$ satisfying the condition $\rank_2(Cl(L))\geq 2$  in Theorem \ref{main2}, there are many such $p$. Indeed, for $p<1000000$, $1852/13099$ of the primes $p\equiv 2 \bmod 9$ and $1629/13068$ of the primes $p\equiv 5\bmod 9$ satisfy the condition. Some notable numerical examples can be found in Appendix \ref{appendixB}.  We note also that there are infinitely many elliptic curves over $\BQ$ with rank $0$ and non-trivial $2$-part of the Tate--Shafarevitch group (see, for example, \cite{razar2}). However, in the rank $1$ case, it seems that one can so far only find elliptic curves over $\BQ$ with non-trivial $2$-part of the Tate--Shafarevich group via numerical computations. Theorem \ref{main2} thus sheds new theoretical light on the existence of many curves with rank $1$ and non-trivial $2$-part of the Tate--Shafarevich group.

\medskip

In the remainder of this introduction, we explain the structure of the paper and the key ideas involved. In Section \ref{section2}, we will introduce the notion of `explicit modulo $3$ Birch--Swinnerton-Dyer conjecture' and show congruences similar to those in Theorem \ref{main1} for the curves $C_p$ and $C_{p^2}$. The proof of Theorem \ref{main1} will be given in Section \ref{section3}. The main idea of the proof is an extension of Zhao's averaging method \cite{zhao, yuka, qiu} from $2$-adic to $3$-adic. However, a direct application of Zhao's method only gives us the $3$-adic valuation of the sum of the two algebraic $L$-values in Theorem \ref{main1}, and does not give us the $3$-adic valuation of the individual algebraic $L$-value. To overcome this difficulty, we will study the `explicit modulo $3$ Birch--Swinnerton-Dyer conjecture' introduced in Section \ref{section2} to separate the terms in the above sum. A key ingredient in establishing such an `explicit modulo $3$ Birch--Swinnerton-Dyer conjecture' is Zhang's adaptation \cite{szhang} of Razar's method \cite{razar1,razar2} from $2$-adic to $3$-adic. This gives us a precise $3$-adic valuation of certain rational functions on the curve~$A$ which appear naturally in a special value formula (see \cite{coates1}) for $L(C_{2^ip^j},1)$ for $i,j\in \{0,1,2\}$. Combining Zhao's method, the results from Section \ref{section2} and congruences between the $L$-values of $C_{2p}$ and $C_{2^2p^2}$ (resp.\ $C_{2p^2}$ and $C_{2^2p}$) for $p\equiv 5\mod 9$ (resp.\ $p\equiv 2\mod 9$) gives us Theorem \ref{main1}.
In Section \ref{section4}, we will prove Theorem \ref{main2} by a $2$-descent argument. In Appendix \ref{appendixB}, we will give some numerical examples which satisfy the condition  $\rank_2(Cl(L))\geq 2$ in Theorem~\ref{main2}.

\medskip

\textbf{Notation and conventions}

Throughout the paper, we will have $K=\BQ(\sqrt{-3})$ and a fixed embedding of $K$ into $\BC$.  We write $\CO_K$ for the integer ring of $K$, and for every ideal $(b)=b\CO_K$ of $K$, $K(b)$ will denote the ray class field over $K$ modulo $(b)$. An odd prime~$p$ is always assumed satisfying the condition \eqref{bcong}, and we write $H_p=K(\sqrt[3]{p})$ and $\CH_p=K(\sqrt[3]{p},\sqrt[3]{2})$. If $m\geq 1$ is an integer, $\mu_m$ will denote the group of $m$-th roots of unity. For any cube free integer $n\geq 3$, we write $C_n$ for the elliptic curve $x^3+y^3=n$.  Let $A:y^2=4x^3-1$ with a period lattice $L=\Omega \CO_K$, where $\Omega=3.059908...$ is a real number. For each integer $n\geq 1$, we set $\Omega_n=\frac{\Omega}{(\sqrt{3}n^{\frac{1}{3}})}$. Let $L(C_n,s)$ be the complex $L$-series of $C_n$ over $\BQ$. We call $\frac{L(C_n,1)}{\Omega_n}$ the algebraic part of $L$-value of $C_n$ at $s=1$. As usual, $\BQ_3$ will denote the field of $3$-adic numbers, and $\BZ_3$ the ring of $3$-adic integers in $\BQ_3$. We write $\ov{\BQ}_3$ for a fixed algebraic closure of $\BQ_3$, and $\ov{\BZ}_3$ will denote the integer ring of $\ov{\BQ}_3$. Given $\alpha\in \ov{\BQ}_3$, we write $\alpha\in 3^{\epsilon_0}(\epsilon_1+3^{\epsilon_2}\ov{\BZ}_3)$ for three non-negative rational numbers $\epsilon_i$ $(i=0,1,2)$ if the number $(\alpha/3^{\epsilon_0})-\epsilon_1$ is divisible by $3^{\epsilon_1}$ in $\ov{\BZ}_3$. Finally, we fix a normalized additive $3$-adic valuation $\ord_3$ on $\ov{\BQ}_3$ such that $\ord_3(3)=1$.

\section{Explicit Modulo $3$ Birch--Swinnerton-Dyer Conjecture for $C_p$ and $C_{p^2}$}\label{section2}

We first introduce a special value formula for $C_n$ to which we will refer frequently in the later sections of the paper. We assume from now on that $n$ is prime to $3$. Let $\psi_n$ be the Hecke character over $K$ associated to $C_n$. From an explicit description of $\psi_n$ contained in \cite{silverman}, we have
\begin{equation}\label{hecke-p}
  \psi_n((\alpha))=\ov{\left(\frac{n}{\alpha}\right)}_3 \alpha
\end{equation}
for every $\alpha\in\CO_K$ which is prime to $3n$ and is congruent to $1$ modulo $3$. Here, $\left(\frac{\cdot}{\cdot}\right)_3$ denotes the cubic residue symbol, and we denote by $(f)$ the conductor of $\psi_n$. A precise definition of the cubic residue symbol and a detailed description of $(f)$ can be found in \cite{stephens}.

Let $z$ and $s$ be complex variables. Given any lattice $\CL$ in the complex plane $\BC$, the Kronecker--Eisenstein series is defined by
\[H_1(z,s,\CL):=\sum_{w\in \CL}\frac{\ov{(z+w)}}{|z+w|^{2s}},\]
where the sum is taken over all $w\in \CL$ except $-z$ if $z\in \CL$. This series converges to define a holomorphic function in $s$ on the half plane $\Re(s)>\frac{3}{2}$, and it has analytic continuation to the whole complex $s$-plane. We define the non-holomorphic Eisenstein series by
\[\CE^*_1(z,\CL)=H_1(z,1,\CL).\]
It is shown in \cite{stephens} that the \Neron differential lattice of $C_n$ is $\Omega_n\CO_K$. We have the following special value formula for $C_n$, whose detailed proof can be found in \cite[Theorem 60]{coates1}.

\begin{thm}
Let $(f)$ be the conductor of $\psi_n$, and let $g$ be an integer in $K$ such that $g$ is a multiple of $f$. Then the value $\CE^*_1\left(\frac{\Omega_n}{g},\Omega_n\CO_K\right)$ lies in $K(g)$, and we have
\begin{equation}\label{coates-formula}
L^{(g)}(C_n,1)=\frac{\Omega_n}{g}{\rm Tr}_{K(g)/K}\left(\CE^*_1\left(\frac{\Omega_n}{g},\Omega_n\CO_K\right)\right),
\end{equation}
where $\mathrm{Tr}_{K(g)/K}$ denotes the trace map from $K(g)$ to $K$.
Here, we denote by $L^{(g)}(C_n,s)$ the imprimitive $L$-series of $C_n$ obtained by omitting the Euler factors at all primes dividing $(g)$, so that in particular in the case $(f)=(g)$, $L^{(g)}(C_n,s)$ is the complex $L$-series for $C_n$.
\end{thm}

\medskip

Let us say from now on that the explicit modulo $3$ Birch--Swinnerton-Dyer conjecture holds for $C_n$ if the algebraic part of its $L$-value at $s=1$, when divided by the product of the Tamagawa factors, satisfies the modulo $3$ congruence predicted by the Birch--Swinnerton-Dyer conjecture. When $L(C_n,1)\neq 0$, the explicit modulo $3$ Birch--Swinnerton-Dyer conjecture is stronger than the $3$-part of the Birch--Swinnerton-Dyer conjecture for $C_n$, since the argument from Iwasawa theory due to Rubin \cite{rubin91} excludes both the $2$-part and the $3$-part of the Birch--Swinnerton-Dyer conjecture for $C_n$. In particular, it is not difficult to check that the congruences in Theorem \ref{main1} gives the explicit modulo $3$ Birch--Swinnerton-Dyer conjecture for $C_{2p}$ and $C_{2^2p^2}$ when $p\equiv 5\bmod 9$ and for $C_{2^2p}$ and $C_{2p^2}$ when $p\equiv 2\bmod 9$. Indeed, the factor $3$ in the denominators of the congruences in Theorem \ref{main1} comes from the Tamagawa factor at the prime $3$ (a description of the Tamagawa factor for $C_n$ is given in \cite{zk87}). Furthermore, by a $3$-descent argument, we can show that the $3$-part of the Tate--Shafarevich group of the elliptic curves in Theorem \ref{main1} is trivial (see Proposition \ref{3-descent}). It follows from \cite{rubin87} that these Tate--Shafarevich groups are finite, and thus by Cassels' theorem \cite{cassels62}, they have square orders. Hence their orders are congruent to $1$ modulo $3$, as required.

In this section, we first prove
\[L(C_p,1)\neq 0  \; \text{ and } \; L(C_{p^2},1)\neq0,\]
then show that the $3$-part of the Birch--Swinnerton-Dyer conjectures holds for these curves. Moreover, we will show that the explicit modulo $3$ Birch--Swinnerton-Dyer conjecture holds for $C_p$ and $C_{p^2}$. We remark that the non-vanishing of the $L$-values of these curves and the $3$-part of the Birch--Swinnerton-Dyer conjecture were shown by Zhang in \cite{szhang}. The main result of this section is the explicit modulo $3$ Birch--Swinnerton-Dyer conjecture for the curves $C_p$ and $C_{p^2}$, which we will prove by applying Zhao's induction argument and by establishing certain congruences between the algebraic $L$-values of $C_p$ and $C_{p^2}$.

\medskip

We write $D$ for $p$ or ${p^2}$. We recall that $A$ has the periods lattice $L=\Omega\CO_K$, and write $\wp(u)$ for the Weierstrass $\wp$-function associated to the lattice $L$. We will use the following adaptation of formula \eqref{coates-formula}. This formula was already obtained in \cite{stephens}, but we give a different proof, and the style of the proof will play a key role in our later arguments.

\begin{prop}\label{ste-f1}
We have
\[L(C_D,1)=\frac{-\Omega}{2\sqrt{3} p}\sum_{c\in \mathfrak{C}}\left(\frac{c}{D}\right)_3 \frac{1}{\wp\left(\frac{c\Omega}{p}\right)-1},\]
where $\mathfrak{C}$ denotes a set of representatives of $(\CO_K/p\CO_K)^\times$ such that $c\in \mathfrak{C}$ implies $-c\in \mathfrak{C}$.
\end{prop}

\begin{proof}
From \cite{stephens}, we know that the conductor of $\psi_D$ divides $(3p)$ and has the same prime ideal divisors as $(3p)$. By formula \eqref{coates-formula}, we have
\begin{equation}\label{2-1}
L(C_D,1)=\frac{\Omega_D}{-3p}\mathrm{Tr}_{K(3p)/K}\left(\CE^*_1\left(\frac{\Omega_D}{-3p},\Omega_D\CO_K\right)\right).
\end{equation}

We make a choice for a set of representatives of integral ideals in $K$ whose Artin symbols give precisely the Galois group $\Gal(K(3p)/K)$. Note that
\[\Gal(K(3p)/K)\simeq \left(\CO_K/3p\CO_K\right)^\times/\mu_6\simeq (\CO_K/p\CO_K)^\times,\]
where the inverse of the second isomorphism is given by sending $c\in (\CO_K/p\CO_K)^\times$ to $3c-p$. Here, we use the condition $-p\equiv 1\mod 3$ and identify $\mu_6$ with $(\CO_K/3\CO_K)^\times$.
Therefore, the Artin symbols of $(3c-p)$ in $\Gal(K(3p)/K)$ give a set of representatives of $\Gal(K(3p)/K)$ as $c$ runs over all the elements of $(\CO_K/p\CO_K)^\times$. Thus from \eqref{2-1} and the explicit Galois action on $\CE^*_1\left(\frac{\Omega_D}{-3p},\Omega_D\CO_K\right)$ described in \cite{coates1}, we obtain
\[\begin{aligned}L(C_D,1)=&\frac{\Omega_D}{-3p} \sum_{c\in (\CO_K/p\CO_K)^\times}\CE^*_1\left(\frac{\psi_D((3c-p))\Omega_D}{-3p},\Omega_D\CO_K\right)\\
=& \frac{\Omega}{-3p} \sum_{c\in (\CO_K/p\CO_K)^\times}\CE^*_1\left(\frac{\psi_D((3c-p))\Omega}{-3p},L\right)
\end{aligned}\]
where $L=\Omega \CO_K$. For the second equality, we use the homogeneity of Eisenstein series
\[\CE^*_1(\lambda z,\lambda L)=\lambda^{-1}\CE^*_1(z, L)\qquad \lambda\in \BC^\times\]
of degree $-1$. Furthermore, by \eqref{hecke-p}, we have
\[\psi_D((3c-p))=\ov{\left(\frac{D}{3c-p}\right)}_3 (3c-p)=\ov{\left(\frac{c}{D}\right)_3}(3c-p).\]
In the second equality, we use the cubic reciprocity law and the fact that $\left(\frac{3}{D}\right)_3=1$ (since $3,D\in\BQ$).
It follows that
\[L(C_D,1)=\frac{\Omega}{-3p} \sum_{c\in (\CO_K/p\CO_K)^\times} \left(\frac{c}{D}\right)_3\CE^*_1\left(\frac{c\Omega}{-p}+\frac{\Omega}{3},L\right).\]

We write $L=\BZ u+\BZ v$ with $\Im(v/u)>0$. We define the positive real number $A(L)$ by $A(L)=\frac{\ov{u}v-u\ov{v}}{2\pi i}$.
Furthermore, we write
\[s_2(L)=\lim_{s>0, s\ra 0}\sum_{w\in L\bs\{0\}} w^{-2}|w|^{-2s}.\]
Let $\zeta(z,L)$ be the Weierstrass zeta function attached to $L$. Then by \cite[Theorem 55]{coates1}, we have
\[\CE^*_1(z,L)=\zeta(z,L)-z s_2(L)-\ov{z} A(L)^{-1}.\]

Note that by \cite[pp. 390--392]{qiu}, we have $A(L)=\sqrt{3}\Omega^2/(2\pi)$, $s_2(L)=\frac{2}{\Omega}\zeta\left(\Omega/2,L\right)-\frac{2\pi}{\sqrt{3}\Omega^2}$ and $\zeta(\Omega/3,L)=\frac{2\pi}{3\sqrt{3}\Omega}+\frac{\sqrt{3}}{3}$. Hence the additive formula of the Weierstrass zeta function gives
\[\zeta\left(\frac{c\Omega}{-p}+\frac{\Omega}{3},L\right)=\zeta\left(\frac{c\Omega}{-p},L\right)+\frac{2\pi}{3\sqrt{3}\Omega}+\frac{1}{\sqrt{3}}+
\frac{1}{2}\left(\frac{\wp'(c\Omega/(-p))-\wp'(\Omega/3)}{\wp(c\Omega/(-p))-\wp(\Omega/3)}\right).\]
From \cite{stephens}, we know that
\[\wp(\Omega/3)=1,\quad \wp'(\Omega/3)=-\sqrt{3},\]
and we also have $\zeta(\Omega/2,L)=\pi/(\sqrt{3}\Omega)$ (see \cite[p. 391]{qiu}).
Thus from all of the above, we obtain
\begin{small}
\begin{align*}L(C_D,1)=\frac{\Omega}{-3p}\sum_{c}\left(\frac{c}{D}\right)_3\left(\zeta\left(\frac{c\Omega}{-p},L\right)+\frac{1}{2}\cdot \frac{\wp'(c\Omega/(-p))+\sqrt{3}}{\wp(c\Omega/(-p))-1}+\frac{1}{\sqrt{3}}+\frac{2\pi\ov{c}}{\sqrt{3}\Omega p}\right),
\end{align*}
\end{small}
where $c$ runs over the elements of $(\CO_K/p\CO_K)^\times$. Now, since $p$ is an odd prime, we can choose a set $\mathfrak{C}$ of representatives of $(\CO_K/p\CO_K)^\times$ such that $c\in \mathfrak{C}$ implies $-c\in \mathfrak{C}$. Note that
\[\sum_{c\in \mathfrak{C}}\left(\frac{c}{p}\right)_3=0,\]
and that $\zeta(z,L)$ and $\wp'(z)$ are odd functions. Noting also that $\left(\frac{-1}{D}\right)_3=1$, we obtain
\[L(C_D,1)=-\frac{\Omega}{2\sqrt{3}p}\sum_{c\in \mathfrak{C}}\left(\frac{c}{D}\right)_3 \frac{1}{\wp\left(\frac{c\Omega}{p}\right)-1}\]
as required.
\end{proof}

The curve $A$ can also be written in the form
\[y^2=x^3-16.\]

Now, we will construct a model $\CA$ of the curve $A$ over a certain finite abelian extension of $K$ so that $\CA$ has good reduction at $3$. The main theoretical reason is the following theorem. The proof of the following theorem can be found in \cite[Theorem 2]{CW} or \cite[Theorem 2.4]{CL}.

\begin{thm}\label{cw22}
Let $E$ be an elliptic curve over $K$ with complex multiplication by $\CO_K$. Let $\fp=(\alpha)$ be any prime ideal of $K$ at which $E$ has good reduction. Denote by $E_\fp$ the subgroup of points in $E(\ov{K})$ which are contained in the kernel of the endomorphism $\alpha$. Then $E$ has good reduction everywhere over $K(E_\fp)$.
\end{thm}

Thus we can construct a model of $A$ over the field $\CK=K(A_{(2-\sqrt{-3})})$, where $(2-\sqrt{-3})$ is a prime ideal of $K$ lying above the prime ideal $7\BZ$ of $\BQ$.

\begin{lem} Let $\CK$ be given as above. Then we have
\[\CK=K\left(\sqrt[6]{\frac{6\sqrt{-3}}{1+3\sqrt{-3}}}\right).\]
\end{lem}

\begin{proof}
Since $A$ has good reduction at $7$, its formal group at $7$ is a Lubin--Tate formal group. Therefore, the degree of $\CK$ over $K$ is equal to $6$, and $\Gal(\CK/K)$ is a cyclic Galois group. Since $\mu_6$ is contained in $K$, by Kummer theory, we just need to find an element $\alpha$ in $K$ such that $\sqrt[6]{\alpha}$ generates $\CK$. One can do a direct calculation using the additive law and complex multiplication on the curve $y^2=x^3-16$ to get such an $\alpha$, and we leave the details to the reader.
\end{proof}

Now, we make the change of variables $y=u^3Y , \;x=u^2X+r$ to the curve $y^2=x^3-16$, where
\begin{equation}\label{r} u=\sqrt[6]{\frac{6\sqrt{-3}}{1+3\sqrt{-3}}}\quad \text{ and }\quad  r=4 \frac{u^2}{\sqrt{-3}}.
\end{equation}

Then we obtain a model
\begin{equation}\label{huaa}
  \CA\colon Y^2=X^3-4\sqrt{-3}X^2-16X-8+8\sqrt{-3}
\end{equation}
of $A$. Note that since $\ord_3(u)=\frac{1}{4}$, $\CA$ has good reduction at $3$.

\begin{lem}\label{val-pt}
If $Q$ is a point on $\CA$ of order prime to $3$, then either $X(Q)=0$ or $\ord_3\left(X(Q)\right)=0$.
\end{lem}
\begin{proof}
The following proof is essentially an analogue of the same result in \cite{razar1}. One can also find a proof in \cite{szhang}, but we give it here for the convenience of the reader. Denote by $\tilde{\CA}$ the reduced curve of $\CA$ modulo $3$, so that
\[\tilde{\CA}\colon \widetilde{Y}^2=\widetilde{X}^3-\widetilde{X}+1.\]
Since $\CA$ has good reduction at $3$, the torsion points of order prime to $3$ on $\CA$ inject into $\tilde{\CA}$, and if $Q_1\neq\pm Q_2$ on $\CA$, we have
\begin{equation}\label{red}
  \widetilde{X(Q_1)-X(Q_2)}=\widetilde{X(Q_1)}-\widetilde{X(Q_2)}\neq 0.
\end{equation}
Let $Q\neq 0$ be a point on $\CA$ satisfying $X(Q)=0$. Then a direct computation shows $X(\sqrt{-3}Q)=X(2Q)$, so that $Q$ is a $(2-\sqrt{-3})$- or a $(2+\sqrt{-3})$-torsion point. Let $Q_1$ be any point with $X(Q_1)\neq 0$. Then clearly $Q_1\neq \pm Q$ since $X(Q)=X(-Q)$. Thus setting $Q_2=Q$ in \eqref{red}, we see that if $X(Q_1)\neq 0$ then $\widetilde{X(Q_1)}\neq 0$, that is, $\ord_3(X(Q_1))=0$.
\end{proof}

From now on, we write $(\wp(z),\wp'(z))$ for a point on $A\colon y^2=4x^3-1$, $(x,y)$ for a point on the model $y^2=x^3-16$ of $A$, and $(X,Y)$ for a point the model $\CA$ whose equation is given in \eqref{huaa}. The relation between the first two is given by
\begin{equation}\label{x-wp}
  x=4\wp(z)\qquad y=4\wp'(z).
\end{equation}

Let $Q=\left(\wp\left(\frac{\Omega}{\Delta}\right), \wp'\left(\frac{\Omega}{\Delta}\right)\right)$ be a point on $A$, where $\Delta$ is an integer prime to $3$. If $c\in\CO_K$ is prime to $\Delta$, then since $x=u^2X+r$ and $\ord_3(u)=1/4$, Lemma \ref{val-pt} gives
\begin{equation}\label{f200}
  4\wp\left(\frac{c\Omega}{\Delta}\right)=x([c]Q)\equiv r\mod 3^{\frac{1}{2}}.
\end{equation}
Here, we denote by $[c]$ the complex multiplication by $c\in\CO_K$.
Note that $r$ is a $3$-adic unit which is independent of the point $[c]Q$, and thus $\wp\left(\frac{c\Omega}{\Delta}\right)$ is a $3$-adic unit.

The following lemma will be repeatedly used in the remainder of this section and in Section \ref{section3}.

\begin{lem}\label{key-val}
Let $\Delta$ be an integer prime to $3$. Then we have
\begin{enumerate}
         \item $\wp\left(\frac{c\Omega}{\Delta}\right)-1\equiv 3^{\frac{1}{3}}u_1\mod 3^{\frac{1}{2}}$, and
         \item $\wp\left(\frac{c\Omega}{\Delta}\right)^3-1\equiv 3u^3_1\mod 3^{\frac{7}{6}}$
 \end{enumerate}
for a $\:3$-adic unit $u_1$ which is independent of the point $[c]Q$. Moreover, we have $u^3_1=(r-1)^3/3$, where $r$ is given in \eqref{r}. We remark that ``$\mod 3^\varepsilon$" (for $\varepsilon>0$) is taken in the ring of integers $\ov{\BZ}_3$ of $\ov{\BQ}_3$.
\end{lem}

\begin{proof}
Noting that $x=u^2X+r$ and $\ord_3(u)=\frac{1}{4}$, Lemma \ref{val-pt} gives the equality
\[x([c]Q)-4= x([c]Q)-1-3\equiv r-1\mod 3^{\frac{1}{2}}.\]
Recall from \eqref{r} that $r=4\frac{u^2}{\sqrt{-3}}$. Thus, we have
\begin{equation}\label{r-1mod3}
  r-1\equiv -\sqrt[3]{14}^{-1}\left(\sqrt[3]{1-3\sqrt{-3}}+\sqrt[3]{14}\right)\mod 3.
\end{equation}
By considering the third power of $\sqrt[3]{1-3\sqrt{-3}}+\sqrt[3]{14}$, we know that $\sqrt[3]{1-3\sqrt{-3}}+\sqrt[3]{14}$ has $3$-adic valuation equal to $\frac{1}{3}$. Therefore, defining $u_1=\frac{r-1}{3^{\frac{1}{3}}}$ and noting that $4\wp\left(\frac{c\Omega}{\Delta}\right)=x([c]Q)$ and $4^{-1}\in 1+3\BZ_3$,
we obtain the first assertion.

For the second assertion, let $\omega=\frac{-1+\sqrt{-3}}{2}$ be a primitive $3$rd root of unity. Note that $\ord_3(\omega^i-1)=\frac{1}{2}$ for $i=1,2$, so that the first assersion gives
\[\wp\left(\frac{c\Omega}{\Delta}\right)-\omega^i\equiv \wp\left(\frac{c\Omega}{\Delta}\right)-1\equiv 3^{\frac{1}{3}}u_1\mod3^{\frac{1}{2}}\; \;\quad\text{ for } i=1,2.\]
Thus
\[\begin{aligned}\wp\left(\frac{c\Omega}{\Delta}\right)^3-1= & \prod^2_{j=0}(3^{\frac{1}{2}}A_j+3^{\frac{1}{3}}u_1)\\
\equiv& 3u^3_1\mod 3^\frac{7}{6}.
\end{aligned}\]
Here, $A_j$ ($j=0,1,2$) are $3$-adic integers. The second assertion now follows.
\end{proof}

\begin{thm}[Zhang \cite{szhang}]\label{zhang1}
Let $p\equiv 2,5\mod 9$ be an odd prime, and recall that $D=p$ or $p^2$. Then
\[L(C_D,1)\neq 0\]
and $\ord_3\left(\frac{L(C_D,1)}{\Omega_D}\right)=0$. Moreover, the $3$-part of the Birch--Swinnerton-Dyer conjecture holds for $C_D$.
\end{thm}

\begin{proof}
The following proof is contained in \cite{szhang}, but we give it here for the convenience of the reader. By Proposition \ref{ste-f1}, we have
$$
  \frac{2p\sqrt{3}}{\Omega} L(C_D,1)=-\sum_{c\in \mathfrak{C}}\left(\frac{c}{D}\right)_3\frac{1}{\wp\left(\frac{c\Omega}{p}\right)-1}.
$$
Now, we define
\[B_k=\left\{c\in \mathfrak{C}: \left(\frac{c}{D}\right)_3=\omega^k\right\}\]
where $k=0,1,2$. We use the convention that $k+1,k+2$ are also in $\{0,1,2\}$ via taking modulo $3$ in the following argument.

We define
\[S_k=\sum_{c\in B_k}\frac{1}{\wp\left(\frac{c\Omega}{p}\right)-1}.\]
Note that $\left(\frac{\omega}{p}\right)_3=\omega$ or $\omega^2$ according as $p\equiv 2\mod 9$ or $p\equiv 5\mod 9$, and that $\omega \wp\left(\frac{c\Omega}{p}\right)=\wp\left(\frac{\omega c\Omega}{p}\right)$. Then
\[\omega B_k=\begin{cases}B_{k+1} &\mbox{if $D=p$ and $p\equiv 2\mod 9$ or if $D=p^2$ and $p\equiv 5\mod 9$,}\\
B_{k+2}&\mbox{if $D=p$ and $p\equiv 5\mod 9$ or if $D=p^2$ and $p\equiv2\mod 9$.}
\end{cases}
\]

We just give the details of the proof in the cases $D=p$ with $p\equiv 2\mod 9$ and $D=p^2$ with $p\equiv 5\mod 9$. The other cases can be proven in the same way.
In these cases, we have
\[\begin{aligned}-\frac{2p\sqrt{3}}{\Omega} L(C_D,1)
=&\sum_{c\in B_0}\left(\frac{1}{\wp\left(\frac{c\Omega}{p}\right)-1}+\frac{\omega}{\omega \wp\left(\frac{c\Omega}{p}\right)-1}+\frac{\omega^2}{\omega^2 \wp\left(\frac{c\Omega}{p}\right)-1}\right)\\
=&\sum_{c\in B_0}\frac{3\wp\left(\frac{c\Omega}{p}\right)^2}{\wp\left(\frac{c\Omega}{p}\right)^3-1}.
\end{aligned}\]
Now, by \eqref{f200} and Lemma \ref{key-val}, we have
\[3\wp\left(\frac{c\Omega}{p}\right)^2\equiv 3r^2\mod 3^{\frac{3}{2}}, \qquad \wp\left(\frac{c\Omega}{p}\right)^3-1\equiv 3u^3_1\mod 3^{\frac{7}{6}}.\]
Recalling that $u_1$ and $r$ are independent of $c$, we obtain
$$-\frac{2p\sqrt{3}}{\Omega} L(C_D,1)\equiv  (r^2u^{-3}_1)\cdot \#(B_0)\mod 3^{\frac{1}{6}}
$$
which is a $3$-adic unit, since $\#(B_0)=\frac{p^2-1}{3}$ is prime to $3$ when $p\equiv 2, 5\mod 9$.
Noting that
\[\Omega_D=\frac{\Omega}{\sqrt{3}D^{\frac{1}{3}}}\]
and $D$ is prime to $3$, we obtain $\ord_3\left(\frac{L(C_D,1)}{\Omega_D}\right)=0$. The theorem follows since we know that the $3$-part of $\Sha(C_D)$ is trivial by $3$-descent (see \cite{szhang}).
\end{proof}

It follows from Theorem \ref{zhang1} of Zhang that $\frac{L(C_D,1)}{\Omega_D}\in \BZ^\times_3$, and we can study its value modulo $3$.
Now, we can prove the explicit modulo $3$ Birch--Swinnerton-Dyer conjecture for $C_D$.

\begin{lem}\label{1mod3}
\noindent
\begin{enumerate}
  \item If $p\equiv 5\mod 9$, then
  \[\frac{L(C_{p^2},1)}{2\Omega_{p^2}}\equiv \frac{L(C_p,1)}{\Omega_p}\mod 3.\]
  \item If $p\equiv 2\mod 9$, then
  \[\frac{L(C_{p^2},1)}{\Omega_{p^2}}\equiv \frac{L(C_p,1)}{2\Omega_p}\mod 3.\]
\end{enumerate}
\end{lem}

We remark that the factor $2$ in the denominators of the above congruences comes from the Tamagawa factor of $C_D$ at $3$.

\begin{proof}
We just give the details of the proof for (1), and (2) can be obtained similarly. By Proposition \ref{ste-f1} and the definition of $\Omega_d$, we have
\[\frac{L(C_{p^2},1)}{2\Omega_{p^2}}=-\frac{\sqrt[3]{p^2}}{4p}\sum_{c\in \mathfrak{C}}\left(\frac{c}{p^2}\right)_3\frac{1}{\wp\left(\frac{c\Omega}{p}\right)-1}\]
and
\[\frac{L(C_{p},1)}{\Omega_{p}}=-\frac{\sqrt[3]{p}}{2p}\sum_{c\in \mathfrak{C}}\left(\frac{c}{p}\right)_3\frac{1}{\wp\left(\frac{c\Omega}{p}\right)-1}.\]

Recall that $\omega=\frac{-1+\sqrt{-3}}{2}$. Since $\ord_3(\omega^i-\omega^j)\geq \frac{1}{2}$ for $i,j=0,1,2$, we obtain
\[\ord_3\left(\left(\frac{c}{p^2}\right)_3-\left(\frac{c}{p}\right)_3\right)\geq \frac{1}{2}.\]
From Lemma \ref{key-val}, we know that
\[\ord_3\left(\wp\left(\frac{c\Omega}{p}\right)-1\right)=\frac{1}{3}.\]
We define
$$B_1:=\sum_{c\in \mathfrak{C}}\left(\frac{c}{p^2}\right)_3\frac{1}{\wp\left(\frac{c\Omega}{p}\right)-1}$$
and
$$B_2:=\sum_{c\in \mathfrak{C}}\left(\frac{c}{p}\right)_3\frac{1}{\wp\left(\frac{c\Omega}{p}\right)-1}.$$
Then
\[\ord_3(B_1-B_2)\geq \frac{1}{6}.\]
Furthermore, we define $A_1:=-\frac{\sqrt[3]{p^2}}{4p}$ and $A_2:=-\frac{\sqrt[3]{p}}{2p}$. Noting that $p\equiv 2\mod 3$, we have
\[\sqrt[3]{p}\equiv 2\mod 3^{\frac{1}{3}},\]
and thus
\[A_1-A_2\equiv0\mod 3^{\frac{1}{3}}. \]
Now, by the relations
\[\frac{L(C_{p^2},1)}{2\Omega_{p^2}}=A_1B_1,\qquad \frac{L(C_p,1)}{\Omega_p}=A_2B_2\]
and the identity
\[A_1B_1-A_2B_2=(A_1-A_2)B_1+A_2(B_1-B_2),\]
we obtain that $3^{\frac{1}{6}}$ divides
\[\frac{L(C_{p^2},1)}{2\Omega_{p^2}}-\frac{L(C_p,1)}{\Omega_p}.\]
Our assertion now follows on noting that these $L$-values are in $\BZ_3$.
\end{proof}

From Theorem \ref{zhang1}, we know that the residue class of the two congruences in Lemma \ref{1mod3} are non-zero modulo $3$. Now, we use Zhao's induction argument to obtain the main result of this section giving the precise residue class modulo $3$ in which these $L$-values lie.

\begin{thm}\label{mod3-bsd}
\noindent
\begin{enumerate}
  \item For $p\equiv 5\mod 9$, we have
\[\frac{L(C_p,1)}{\Omega_p}\equiv \frac{L(C_{p^2},1)}{2\Omega_{p^2}}\equiv 1\mod 3.\]
Moreover, the explicit modulo $3$ Birch--Swinnerton-Dyer conjecture holds for $C_p$ and $C_{p^2}$.
  \item For $p\equiv 2\mod 9$, we have
\[\frac{L(C_p,1)}{2\Omega_p}\equiv \frac{L(C_{p^2},1)}{\Omega_{p^2}}\equiv 1\mod 3.\]
Moreover, the explicit modulo $3$ Birch--Swinnerton-Dyer conjecture holds for $C_p$ and $C_{p^2}$.
\end{enumerate}
\end{thm}

For simplicity, we denote by $\delta$ the period $\Omega_p$ of $C_p$ in the following. We will just give the details for the case $p\equiv 5\mod 9$, and the other case is similar. Recall that $K(3p)$ is the ray class field of $K$ modulo $(3p)$. Then we have $\sqrt[3]{p}\in K(3p)$, and we recall that $H_p$ is the field $K(\sqrt[3]{p})$.

For any positive divisor $d\mid (p)^2$, we know that the period relation
\begin{equation}\label{f203}
  \Omega_d=\delta\sqrt[3]{(p/d)}
\end{equation}
holds.

By formula \ref{coates-formula}, and the period relation \eqref{f203}, we have
$$
  \frac{L^{(3p)}(C_d,1)}{\delta\sqrt[3]{p}}=\sum_{\sigma\in \Gal(K(3p)/K)}(\sqrt[3]{d})^{\sigma-1}\frac{1}{3p}\CE^*_1\left(\frac{\delta\sqrt[3]{p}}{3p},\delta\sqrt[3]{p}\CO_K\right)^\sigma.
$$
Note that $\sum_d (\sqrt[3]{d})^{\sigma-1}=3$ if $\sigma$ fixes $\sqrt[3]{d}$, and  $\sum_d (\sqrt[3]{d})^{\sigma-1}=0$ otherwise. Now, summing the above formula over all $d\mid(p)^2$ and noting $\sqrt[3]{p}\in H_p$, we obtain
\begin{equation}\label{p-f3}
  \sum_{d\mid(p)^2}\frac{L^{(3p)}(C_d,1)}{\delta}=3\mathrm{Tr}_{K(3p)/H_p}\left(\frac{1}{3p}\CE^*_1\left(\frac{\delta}{3p},\delta\CO_K\right)\right).
\end{equation}

First, we study the terms on the left hand side of \eqref{p-f3}. We write $u$ ($=1\textrm{ or }2$) for the residue
class of $\frac{L(C_p,1)}{\Omega_p}\mod 3$. Then we have

\begin{lem}\label{p-left-3}
\noindent
\begin{enumerate}
  \item For $d=1$, we have $\frac{L^{(3p)}(C,1)}{\delta}=\frac{\sqrt[3]{p}(p+1)}{3p}\in (\sqrt[3]{p}+3\ov{\BZ}_3)= 2+3^{\frac{1}{3}}\ov{\BZ}_3$.
  \item For $d=p$, we have $\frac{L^{(3p)}(C_p,1)}{\delta}=\frac{L(C_p,1)}{\Omega_p}\in u+3\ov{\BZ}_3$.
  \item For $d=p^2$, we have $\frac{L^{(3p)}(C_{p^2},1)}{\delta}=\frac{L(C_{p^2},1)}{\sqrt[3]{p}\Omega_{p^2}}\in \frac{2u}{\sqrt[3]{p}}+3\ov{\BZ}_3$.
\end{enumerate}
\end{lem}

As for the right hand side of \eqref{p-f3}, we have the following standard calculation. Here, we recall that $\Omega=\delta\sqrt{3}\sqrt[3]{p}$. We will use the
same set $\mathfrak{C}$ of representatives of $\Gal(K(3p)/K)\simeq (\CO_K/p\CO_K)^\times$ as in Proposition \ref{ste-f1}. We
denote by $V$ the subset of $\mathfrak{C}$ such that $c\in V$ if and only if $\left(\frac{c}{p}\right)_3=1$. We see that $H_p$ is precisely the fixed field of the subgroup generated by the Artin symbol of elements in~$V$, and thus $V$ gives a set of representative for $\Gal(K(3p)/H_p)$. Recall that $\psi_p$ is the Hecke character over $K$ associated to $C_p$. Then

\[\begin{aligned}\Psi:=& 3\mathrm{Tr}_{K(3p)/H_p}\left(\frac{1}{3p}\CE^*_1\left(\frac{\delta}{3p},\delta\CO_K\right)\right)\\
=&-\frac{1}{p}\mathrm{Tr}_{K(3p)/H_p}\left(\CE^*_1\left(\frac{\delta}{-3p},\delta\CO_K\right)\right)\\
=&-\frac{1}{p\delta}\cdot\delta\sum_{c\in V}\CE^*_1\left(\frac{\psi_p((3c-p))\delta}{-3p},\delta\CO_K\right)\\
=&-\frac{1}{p\delta}\cdot\Omega\sum_{c\in V}\CE^*_1\left(\frac{\psi_p((3c-p))\Omega}{-3p},\Omega\CO_K\right)\\
=&-\frac{1}{p}\sum_{c\in V}\CE^*_1\left(\frac{\psi_p((3c-p))\Omega}{-3p},\Omega\CO_K\right)\sqrt[3]{p}\sqrt{3}.
\end{aligned}\]

Noting that, by our choice of $\mathfrak{C}$, we have $-c\in V$ whenever $c\in V$, the same argument as in the proof of Proposition \ref{ste-f1} shows
\begin{equation}\label{p-f4}
  \Psi=-\frac{3\sqrt[3]{p}}{2p}\sum_{c\in V}\frac{1}{\wp\left(\frac{c\Omega}{p}\right)-1}-\frac{\sqrt[3]{p}}{p}\#(V),
\end{equation}
where $\#(V)=\frac{p^2-1}{3}$.

Now, from Lemma \ref{key-val}, we know that the first term on the right hand side of \eqref{p-f4} has a positive $3$-adic valuation, while
the second term in \eqref{p-f4} is
\[-\frac{\sqrt[3]{p}}{p}\#(V)=-\frac{\sqrt[3]{p}}{p}\left(\frac{p+1}{3}\right)(p-1)\in 2\sqrt[3]{p}+3\ov{\BZ}_3.\]
Moreover, Lemma \ref{p-left-3} and \eqref{p-f3} give
\begin{equation}\label{p-f5}
\frac{L(C_p,1)}{\delta}+\frac{L(C_{p^2},1)}{\delta}\in \sqrt[3]{p}+3\ov{\BZ}_3=2+3^{\frac{1}{3}}\ov{\BZ}_3.
\end{equation}
On the other hand, by (2) and (3) of Lemma \ref{p-left-3}, the sum in \eqref{p-f5} is contained in
\begin{equation}\label{p-f6}
  (u+3\ov{\BZ}_3)+(\frac{2u}{\sqrt[3]{p}}+3\ov{\BZ}_3)=u(1+\frac{2}{\sqrt[3]{p}})+3\ov{\BZ}_3=2u+3^{\frac{1}{3}}\ov{\BZ}_3.
\end{equation}
Here, we use the fact that
$$\sqrt[3]{p}\equiv 2\mod 3^{\frac{1}{3}}$$
and $3^{\frac{1}{3}}$ exactly divides $\sqrt[3]{p}-2$ when $p\equiv 5\mod 9$. Comparing formulas \eqref{p-f5} and \eqref{p-f6}, we obtain $u=1$. This completes the proof of Theorem \ref{mod3-bsd}.

\section{Explicit Modulo $3$ Birch--Swinnerton-Dyer Conjecture for $C_{2^i p^j}$}\label{section3}

The goal of this section is to give the proof of Theorem \ref{main1}. We prove this in the following steps. Firstly,
we use Zhao's induction argument to show that the sum of algebraic part of $L$-values of $C_{2p}$ and $C_{2^2p^2}$ (resp.\ $C_{2p^2}$ and $C_{2^2p}$)  is non-zero when $p\equiv 5\mod 9$ (resp.\ $p\equiv 2\mod 9$) and satisfies an explicit modulo $3$ congruence. Secondly, by establishing some congruences between the algebraic $L$-values in the sum of the first step, we can show each term in the sum  is non-zero and satisfies the explicit modulo $3$ Birch--Swinnerton-Dyer conjecture. In this section, we again write $D$ for $p$ or $p^2$.
The first key result of this section is the following.

\begin{thm}\label{sum-exp-3bsd}
\noindent
\begin{enumerate}
  \item If $p\equiv 5\mod 9$, we have
\[\frac{L(C_{2p},1)}{\Omega_1}+\frac{L(C_{2^2p^2},1)}{\Omega_1}\in 3(2+3^{\frac{1}{3}}\ov{\BZ}_3)\]
and
\[\frac{L(C_{2p},1)}{\Omega_{2p}}+\frac{L(C_{2^2p^2},1)}{\Omega_{2^2p^2}}\in 3(2+3\BZ_3).\]
  \item If $p\equiv 2\mod 9$, we have
\[\frac{L(C_{2^2p},1)}{\Omega_1}+\frac{L(C_{2p^2},1)}{\Omega_1}\in 3(1+3^{\frac{1}{3}}\ov{\BZ}_3)\]
and
\[\frac{L(C_{2^2p},1)}{\Omega_{2^2p}}+\frac{L(C_{2p^2},1)}{\Omega_{2p^2}}\in 3(2+3\BZ_3).\]
\end{enumerate}
\end{thm}

We will only give the details for the case $p\equiv 5\mod 9$, and leave the proof for the other case to the reader. Before the proof, we remark that, by a root number consideration (for details, see \cite{liv95}), we know that $L(C_{2p^2},1)=L(C_{2^2p},1)=0$ (resp.\ $L(C_{2p},1)=L(C_{2^2p^2},1)=0$) when $p\equiv 5\mod 9$ (resp.\ $p\equiv 2\mod 9$).

For simplicity, we write $\gamma$ for $\Omega_{2p}$ in the following. Then we have the following period relation:
\begin{equation}\label{f300}
 \Omega_d=\gamma\sqrt[3]{\frac{2p}{d}}.
\end{equation}
We recall that $\CH_p=K(\sqrt[3]{2},\sqrt[3]{p})$.

\begin{lem}\label{a00}
We have $\sqrt[3]{p},\sqrt[3]{2}\in K(6p)$, and the degree of the field extension $[K(6p):\CH_p]=\frac{p^2-1}{3}$, which is prime to $3$.
\end{lem}
\begin{proof}
Since $K$ has class number $1$, given any prime ideal $\fq=\alpha\CO_K$ of $\CO_K$ which is prime to $6p$
and $\alpha\equiv 1\mod 3$, the Artin symbol $\sigma_\fq$ of $\fq$ in $\Gal(K(6p)/K)$ acts on $\sqrt[3]{p}$ via
\[(\sqrt[3]{p})^{\sigma_\fq}=\left(\frac{p}{\alpha}\right)_3\sqrt[3]{p},\]
where $\left(\frac{\cdot}{\cdot}\right)_3$ denotes the cubic residue symbol. The Galois action on $\sqrt[3]{2}$ is given similarly. Thus, for any prime ideal $\fm$ whose generator is congruent to $1$ modulo $6p$,
we know that $\sqrt[3]{p},\sqrt[3]{2}$ are fixed by the Artin symbol $\sigma_\fm$, and the first assertion follows.
For the second assertion, note that $\Gal(K(6p)/K)\simeq (\CO_K/2p\CO_K)^\times$. Hence
$[K(6p): K(\sqrt[3]{p},\sqrt[3]{2})]=\frac{p^2-1}{3}$, which is prime to $3$ by our assumption that $p\equiv 2,5\mod 9$.
\end{proof}

For each $d$ with $0<d\mid (2p)^2$, from \eqref{coates-formula}, and
the period relation \eqref{f300}, we have
$$
  \frac{L^{(6p)}(C_d,1)}{\gamma\sqrt[3]{2p}}=\sum_{\sigma\in \Gal(K(6p)/K)}\frac{1}{6p} (\sqrt[3]{d})^{\sigma-1}\left(\CE^*_1\left(\frac{\gamma\sqrt[3]{2p}}{6p},\gamma\sqrt[3]{2p}\CO_K\right)\right)^\sigma.
$$

Taking the sum over all $0<d\mid (2p)^2$, we obtain the following in the same way we obtained \eqref{p-f3}:

\begin{equation}\label{2p-f3}
  \sum_{d\mid(2p)^2}\frac{L^{(6p)}(C_d,1)}{\gamma\sqrt[3]{2p}}=3^2\mathrm{Tr}_{K(6p)/\CH_p}\left(\frac{1}{6p}\CE^*_1\left(\frac{\gamma\sqrt[3]{2p}}{6p},\gamma\sqrt[3]{2p}\CO_K\right)\right).
\end{equation}

First, we deal with terms on the left hand side of \eqref{2p-f3}. We list the results in the following two lemmas, and leave the proofs to the reader.

\begin{lem}\label{2p-4}
We have
\begin{enumerate}
  \item $\frac{L^{(6p)}(C_1,1)}{\gamma\sqrt[3]{2p}}\in 3(2+3\ov{\BZ}_3)$.
  \item $\frac{L^{(6p)}(C_2,1)}{\gamma\sqrt[3]{2p}}\in 3(\sqrt[3]{2^2}+3\ov{\BZ}_3)$.
  \item $\frac{L^{(6p)}(C_4,1)}{\gamma\sqrt[3]{2p}}\in 3(\frac{1}{\sqrt[3]{2^2}}+3\ov{\BZ}_3)$.
\end{enumerate}
\end{lem}

\bigskip

By Theorem \ref{mod3-bsd}, we have the following estimate.

\begin{lem}\label{2p-p2}
We have
\begin{enumerate}
  \item $\frac{L^{(6p)}(C_p,1)}{\gamma\sqrt[3]{2p}}\in 3\left(\frac{2}{\sqrt[3]{p}}+3\ov{\BZ}_3\right)$.
  \item $\frac{L^{(6p)}(C_{p^2},1)}{\gamma\sqrt[3]{2p}}\in 3\left(\frac{1}{\sqrt[3]{p^2}}+3\ov{\BZ}_3\right)$.
\end{enumerate}
Moreover,
\[\frac{L^{(6p)}(C_p,1)}{\gamma\sqrt[3]{2p}}+\frac{L^{(6p)}(C_{p^2},1)}{\gamma\sqrt[3]{2p}}\in 3\left(\frac{2\sqrt[3]{p}+1}{\sqrt[3]{p^2}}+3\ov{\BZ}_3\right)=3(2+3^{\frac{1}{3}}\ov{\BZ}_3), \]
where we use the fact that $\ord_3(\sqrt[3]{p}-2)=\frac{1}{3}$.
\end{lem}

\begin{lem}\label{2p-phi}
Define
\[\Phi:=\mathrm{Tr}_{K(6p)/\CH_p}\left(\frac{1}{6p}\CE^*_1\left(\frac{\gamma\sqrt[3]{2p}}{6p},\gamma\sqrt[3]{2p}\CO_K\right)\right).\]
Then
\[3^2\Phi\in 3(2+3^{\frac{2}{3}}\ov{\BZ}_3).\]
\end{lem}

\begin{proof} Since $\sqrt[3]{2p}\in \CH_p$, we have
\[\Phi=\frac{1}{\sqrt[3]{2p}}\mathrm{Tr}_{K(6p)/\CH_p}\left(\frac{1}{6p}\CE^*_1\left(\frac{\gamma}{6p},\gamma\CO_K\right)\right).\]
We know that $\Gal(K(6p)/K)$ is isomorphic to
\begin{equation}\label{b06}
  \left(\CO_K/6p\CO_K\right)^\times/\mu_6\simeq \left(\CO_K/2p\CO_K\right)^\times,
\end{equation}
where we identify $\mu_6$ with $\left(\CO_K/3\CO_K\right)^\times$, and the isomorphism from right to left is given by sending any element $c$ of $\left(\CO_K/2p\CO_K\right)^\times$ to $3c+2p$. Here, we note that $2p\equiv 1\mod 3$.

For the set of representatives $\{3c+2p :c\in (\CO_K/2p\CO_K)^\times\}$ of \eqref{b06}, we know that the Artin symbol of the ideal $(3c+2p)$ in $\Gal(K(6p)/K)$ fixes $\sqrt[3]{2}$ and $\sqrt[3]{p}$ if and only if
\[\left(\frac{p}{3c+2p}\right)_3=\left(\frac{c}{p}\right)_3=1\]
and
\[\left(\frac{2}{3c+2p}\right)_3=\left(\frac{c}{2}\right)_3=1\]
hold. Therefore, we can identify
\[V:=\left\{c\in(\CO_K/2p\CO_K)^\times: \left(\frac{c}{p}\right)_3=\left(\frac{c}{2}\right)_3=1 \right\}\]
with the Galois group $\Gal(K(6p)/\CH_p)$. Since $\left(\frac{-1}{c}\right)_3=1$, we can choose a set of representatives of $V$ in $\CO_K$ such that $c\in V$ implies $-c\in V$.

Then by the action of Artin symbols on $\CE^*_1\left(\frac{\gamma}{6p},\gamma\CO_K\right)$ (see the proof in \cite[Theorem 60]{coates1}), we have
\[\Phi=\frac{1}{6p\sqrt[3]{2p}} \sum_{c\in V}\CE^*_1\left(\frac{\psi_{2p}((3c+2p))\gamma}{6p},\gamma\CO_K\right).\]
Recalling the definition of $\psi_{2p}$ in \eqref{hecke-p}, it follows that
\[\Phi=\frac{1}{6p\sqrt[3]{2p}}\sum_{c\in V} \CE^*_1\left(\frac{c\gamma}{2p}+\frac{\gamma}{3},\gamma\CO_K\right).\]
Now, noting that $\Omega=\gamma\sqrt{3}\sqrt[3]{2p}$ and $L=\Omega\CO_K$, we obtain
\[\Phi=\frac{1}{2\sqrt{3} p}\sum_{c\in V}\CE^*_1\left(\frac{c\Omega}{2p}+\frac{\Omega}{3},L\right).\]
By a similar method as in the proof of Proposition \ref{ste-f1}, we have
$$
  \Phi=\frac{1}{4p}\sum_{c\in V}\frac{1}{\wp\left(\frac{c\Omega}{2p}\right)-1}+\frac{1}{6p}\#(V).
$$
Furthermore, from Lemma \ref{key-val}, we know
$$\ord_3\left(\wp\left(\frac{c\Omega}{2p}\right)-1\right)=\frac{1}{3}.$$
The lemma now follows on noting that we have $\#(V)=\frac{p^2-1}{3}$ by Lemma \ref{a00} and $p\equiv 5 \bmod 9$. 
\end{proof}

\medskip

By Lemmas \ref{2p-4} and \ref{2p-p2}, we have
\[\msa:=\sum_{d\mid 4}\frac{L^{(6p)}(C_d,1)}{\gamma\sqrt[3]{2p}}\in 3(2+\sqrt[3]{2^2}+\frac{1}{\sqrt[3]{2^2}}+3\ov{\BZ}_3)=3(1+3^{\frac{1}{3}}\ov{\BZ}_3)\]
and
\[\msb:=\sum_{d\neq 1, d\mid p^2}\frac{L^{(6p)}(C_d,1)}{\gamma\cdot\sqrt[3]{2p}}\in 3(2+3^{\frac{1}{3}}\ov{\BZ}_3).\]
Here, we use the fact
\[2+\sqrt[3]{2^2}+(\sqrt[3]{2^2})^{-1}\in 1+3^{\frac{1}{3}}\ov{\BZ}_3.\]

Now, a root number consideration combined with Lemmas \ref{2p-4}, \ref{2p-p2}, \ref{2p-phi} and formula \eqref{2p-f3} gives
\begin{equation}\label{mention}
  \frac{L(C_{2p},1)}{\Omega_1}+\frac{L(C_{2^2p^2},1)}{\Omega_1}=3^2\Phi-\msa-\msb\in 3(2+3^{\frac{1}{3}}\ov{\BZ}_3),
\end{equation}
where we use the fact that $\sqrt[3]{p}\cdot\gamma=\Omega_1$. Thus, the first equation of Theorem \ref{sum-exp-3bsd} (1) holds. The second equation follows on noting that
the algebraic part of the $L$-values of $C_{2p}$ and $C_{2^2p^2}$ are rational numbers and that these algebraic $L$-values are divisible by $3$ in $\ov{\BZ}_3$, which will be shown in \eqref{div3}. This completes the proof of Theorem \ref{sum-exp-3bsd}.

\bigskip

Now, we will show that each $L$-value in the sum in Theorem \ref{sum-exp-3bsd} is non-zero. This is achieved by establishing certain congruence formulas between the algebraic $L$-values of $C_{2p}$ and $C_{2^2p^2}$ (resp.\ $C_{2^2p}$ and $C_{2p^2}$) for $p\equiv 5\mod 9$ (resp.\ $p\equiv 2\mod 9$). The main result of this section is the following.

\begin{thm}\label{mod3-bsd-2p}
\noindent
\begin{enumerate}
  \item Assume $p\equiv 2\mod 9$, then we have
  \[\frac{L(C_{2^2p},1)}{3\Omega_{2^2p}}\equiv \frac{L(C_{2p^2},1)}{3\Omega_{2p^2}}\equiv 1\mod 3.\]
  \item Assume $p\equiv 5\mod 9$, then we have
  \[\frac{L(C_{2p},1)}{3\Omega_{2p}}\equiv \frac{L(C_{2^2p^2},1)}{3\Omega_{2^2p^2}}\equiv 1\mod 3.\]
\end{enumerate}
\end{thm}

We will give the details of the proof for the case $p\equiv 5\mod9$, and the other case is similar.

\begin{prop}\label{a01}
The element $\sqrt{3}$ is not in the field $K(6p)$.
\end{prop}

\begin{proof}
For the field extensions $\BQ\subset K\subset L$, we have the following conductor-like formula:
\[\fd_{L/\BQ}=\fd^{[K:\BQ]}_{K/\BQ}\cdot N_{K/\BQ}(\fd_{L/K}),\]
where $\fd_{M/J}$ denotes the discriminant ideal of $\CO_J$ for an extension $M/J$ and $\CO_J$ denotes the ring of integer of $J$. Note that $\BQ(\sqrt{3},\sqrt{-3})$ is a splitting field of $x^4+8x^2+4$ over $\BQ$ with discriminant $2^4 3^2$. Thus $K(\sqrt{3})/K$ has conductor $(4)$, and $\sqrt{3}$ cannot be in $K(6p)$.
\end{proof}

\begin{prop}\label{choose-3}
We can choose a set of representatives $\CC$ of $(\CO_K/2p\CO_K)^\times$ in $\CO_K$ such that $c\in\CC$ whenever $-c\in \CC$, and the Artin symbol of $c$ in $\Gal(K(6p)(\sqrt{3})/K)$ fixes $\sqrt{3}$.
\end{prop}

\begin{proof}
By Proposition \ref{a01}, we know the fields $K(\sqrt{3})$ and $K(6p)$ are linearly disjoint over $K$. Thus, given an element $\tau_1\in \Gal(K(6p)/K)$ and the identity element $\iota \in \Gal(K(\sqrt{3})/K)$, the Chebotarev density theorem gives an element $c_1\in \CO_K$ prime to $6p$ whose Artin symbol $\sigma_{c_1}$ in $\Gal(K(6p)(\sqrt{3})/K)$ satisfies $\sigma_{c_1}\mid_{K(6p)}=\tau_1$ and $\sigma_{c_1}\mid_{K(\sqrt{3})}=\iota$.

Now, viewing $c_1\in (\CO_K/2p\CO_K)^\times$ via the isomorphism in \eqref{b06}, we know that $-c_1$ gives another element
in $\Gal(K(6p)(\sqrt{3})/K)$ whose restriction to $K(\sqrt{3})$ is $\iota$, since by class field theory the Artin symbol of $-c_1$ acts on $\sqrt{3}$ via the norm map from $K$ to $\BQ$, so must act trivially on $\sqrt{3}$ (since $c_1$ acts trivially on $\sqrt{3}$). In this way, we can choose the set $\CC$, and the lemma follows.
\end{proof}

From now on, we fix the above choice of the representatives $\CC$. We denote by $V_p$ the Galois group $\Gal(K(6p)/H_p)$. We view $V_p$ as the subset
\[V_p=\{c\in \CC: \left(\frac{c}{p}\right)_3=1\}\]
of $\CC$. Let $V_{2p}$ the subset of $\CC$ giving a set of representatives of the Galois group $\Gal(K(6p)/\CH_p)$.

In the following argument, we will study the partial summation from Zhao's method. For simplicity, we write $t$ for $2$ or $2^2$, and we will study the relations between the $L$-values of $C_{t}$, $C_{tp}$ and $C_{tp^2}$.
We need to use the key result in Theorem \ref{mod3-bsd} on the explicit modulo $3$ Birch--Swinnerton-Dyer conjecture for $C_{p}$ and $C_{p^2}$.

Given any $0<d\mid(p)^2$, we denote by $\eta_t$ the period of $C_{tp}$ for simplicity. Then we have the period relation
\begin{equation}\label{sp-f0}
  \Omega_{td}=\eta_t\cdot\sqrt[3]{\frac{p}{d}}.
\end{equation}

For any such $d$, using
the above period relation \eqref{sp-f0} and formula \eqref{coates-formula}, we have
\begin{equation}\label{sp-f2}
  \frac{L^{(6p)}(C_{td},1)}{\eta_t\sqrt[3]{p}}=\frac{1}{6p}\sum_{\sigma\in \Gal(K(6p)/K)}(\sqrt[3]{d})^{\sigma-1}\CE^*_1\left(\frac{\eta_t
  \sqrt[3]{p}}{6p},\eta_t\sqrt[3]{p}\CO_K\right)^\sigma.
\end{equation}
Now, taking the sum over all $d$ dividing $p^2$ in \eqref{sp-f2}, we obtain
$$
  \sum_{d\mid p^2}\frac{L^{(6p)}(C_{td},1)}{\eta_t\cdot\sqrt[3]{p}}=
  \frac{3}{6p}\mathrm{Tr}_{K(6p)/H_p}\left(\CE^*_1\left(\frac{\eta_t\sqrt[3]{p}}{6p},\eta_t\sqrt[3]{p}\CO_K\right)\right) \qquad t=2,2^2.
$$

\medskip

First, we consider the case $t=2$.

\begin{lem}\label{sp-f6}
We have
\[\frac{L^{(6p)}(C_2,1)}{\eta_t\sqrt[3]{p}}\in 3(2+3\ov{\BZ}_3).\]
\end{lem}

\begin{lem}\label{sp-f7}
We have
\begin{small}
\[\frac{3}{6p}\mathrm{Tr}_{K(6p)/H_p}\left(\CE^*_1\left(\frac{\eta_t\sqrt[3]{p}}{6p},\eta_t\sqrt[3]{p}\CO_K\right)\right)
=\frac{3}{4p}\mathrm{Tr}_{\CH_p/H_p}\left(\sqrt[3]{2}\sum_{c\in V_{2p}}\frac{1}{\wp\left(\frac{c\Omega}{2p}\right)-1}\right).\]
\end{small}
\end{lem}

\begin{proof}
Recall that we use the set $\CC$ of representatives of $\Gal(K(6p)/K)$ chosen in Proposition \ref{choose-3}. We take the trace of $\CE^*_1\left(\frac{\eta_t\sqrt[3]{p}}{6p},\eta_t\sqrt[3]{p}\CO_K\right)$ from $K(6p)$ to $K$ in the following two steps: first from $K(6p)$ to $\CH_p$, then from $\CH_p$ to $H_p$. Noting the Galois action on $\CE^*_1\left(\frac{\eta_t\sqrt[3]{p}}{6p},\eta_t\sqrt[3]{p}\CO_K\right)$ as described in \cite{coates1}, the period relation
\[\eta_t\cdot\sqrt[3]{2p}\cdot\sqrt{3}=\Omega\]
and that $\sqrt{3}$ is fixed by $\CC$, we obtain the following equality after taking the trace map from $K(6p)$ to $\CH_p$:
\begin{footnotesize}
\begin{equation}\label{000001}
 \frac{3}{6p}\mathrm{Tr}_{K(6p)/H_p}\left(\CE^*_1\left(\frac{\eta_t\sqrt[3]{p}}{6p},\eta_t\sqrt[3]{p}\CO_K\right)\right)=\frac{\sqrt{3}}{2p} \mathrm{Tr}_{\CH_p/H_p}\left(\sqrt[3]{2}\sum_{c\in V_{2p}}\CE_1^*\left(\frac{c\Omega}{2p}+\frac{\Omega}{3},\Omega\right)\right).
\end{equation}
\end{footnotesize}
By a similar calculation as in the proof of the Proposition \ref{ste-f1}, we get
\begin{equation}\label{000002}
  \sum_{c\in V_{2p}}\CE_1^*\left(\frac{c\Omega}{2p}+\frac{\Omega}{3},\Omega\right)=\frac{\sqrt{3}}{2}\sum_{c\in V_{2p}}\frac{1}{\wp\left(\frac{c\Omega}{2p}-1\right)}+\frac{\#(V_{2p})}{\sqrt{3}}.
\end{equation}
Note that $\CH_p=H_p(\sqrt[3]{2})$ and $\sqrt{3}$ is fixed by $\CC$. Thus applying \eqref{000002} to \eqref{000001} gives the lemma.
\end{proof}

\medskip

Next, we consider the case of $t=2^2$. Similar to the case of $t=2$, we have
\[\frac{L^{(6p)}(C_{2^2},1)}{\eta_t\sqrt[3]{p}}\in 3(1+3\ov{\BZ}_3)\]
and
\begin{small}
$$
 \frac{3}{6p}\mathrm{Tr}_{K(6p)/H_p}\left(\CE^*_1\left(\frac{\eta_t\sqrt[3]{p}}{6p},\eta_t\sqrt[3]{p}\CO_K\right)\right)
 =\frac{3}{4p}\mathrm{Tr}_{\CH_p/H_p}\left(\sqrt[3]{2^2}\sum_{c\in V_{2p}}\frac{1}{\wp\left(\frac{c\Omega}{2p}\right)-1}\right).
$$
\end{small}
We write $\mse$ for the sum $\sum_{c\in V_{2p}}\frac{1}{\wp\left(\frac{c\Omega}{2p}\right)-1}$. Then we have
\begin{equation}\label{k1}
  \frac{L^{(6p)}(C_2,1)}{\eta_2\sqrt[3]{p}}+\frac{L^{(6p)}(C_{2p},1)}{\eta_2\sqrt[3]{p}}=\frac{3}{4p}\mathrm{Tr}_{\CH_p/H_p}\left(\sqrt[3]{2}\cdot\mse\right)
\end{equation}
and
\begin{equation}\label{k2}
  \frac{L^{(6p)}(C_{2^2},1)}{\eta_{2^2}\sqrt[3]{p}}+\frac{L^{(6p)}(C_{2^2p^2},1)}{\eta_{2^2}\sqrt[3]{p}}=\frac{3}{4p}\mathrm{Tr}_{\CH_p/H_p}\left(\sqrt[3]{2^2}\cdot\mse\right).
\end{equation}

Given any $\sigma\in \Gal(\CH_p/H_p)$, we write $\chi_2(\sigma)=(\sqrt[3]{2})^{\sigma-1}$ and $\chi_2(\sigma)^2=(\sqrt[3]{2^2})^{\sigma-1}$. Note that these characters take values in $\mu_3$.

Now, dividing \eqref{k1} by $\sqrt[3]{2}$ and \eqref{k2} by $\sqrt[3]{2^2}$, we get
\begin{equation}\label{k11}
    \frac{L^{(6p)}(C_2,1)}{\Omega_1}+\frac{L^{(6p)}(C_{2p},1)}{\Omega_1}=\frac{3}{4p}\sum_{\sigma\in \Gal(\CH_p/H_p)}\chi_2(\sigma)\mse^\sigma
\end{equation}
and
\begin{equation}\label{k22}
    \frac{L^{(6p)}(C_{2^2},1)}{\Omega_1}+\frac{L^{(6p)}(C_{2^2p^2},1)}{\Omega_1}=\frac{3}{4p}\sum_{\sigma\in\Gal(\CH_p/H_p)}\chi_2(\sigma)^2\mse^\sigma.
\end{equation}
By Lemma \ref{key-val}, we know that $\ord_3(3\mse^\sigma)=\ord_3((3\mse)^\sigma)\geq \frac{2}{3}$. Thus \eqref{k11} and \eqref{k22} give
\begin{equation}\label{div3}
\ord_3\left(\frac{L^{(6p)}(C_{2^2p^2},1)}{\Omega_1}\right)\geq 1,\quad \ord_3\left(\frac{L^{(6p)}(C_{2p},1)}{\Omega_1}\right)\geq1.
\end{equation}
As we mentioned after \eqref{mention}, \eqref{div3} completes the proof Theorem \ref{sum-exp-3bsd}.

Recall that $\sqrt[3]{p}\equiv \sqrt[3]{2}\equiv 2\mod 3^{\frac{1}{3}}$. Thus dividing both sides of \eqref{k11} and \eqref{k22} by $3$ and subtracting one from the other, and noting that
\[\ord_3(\chi_2(\sigma)-\chi_2(\sigma)^2)\geq\frac{1}{2}> \frac{1}{3}=\ord_3\left(\wp\left(\frac{c\Omega}{2p}\right)-1\right),\]
we see that
\begin{equation}\label{s11}
  \frac{L(C_{2p},1)}{3\Omega_1}-\frac{L(C_{2^2p^2},1)}{3\Omega_1}
\end{equation}
has a positive $3$-adic valuation. Here, we use the fact that $\chi_2(\sigma)-\chi_2(\sigma)^2$ is fixed by the elements of $\Gal(\CH_p/H_p)$ because it takes values in $K$, and that $\ord_3((\chi_2(\sigma)-\chi_2(\sigma)^2)\mse^\sigma)\geq \frac{1}{6}$.

By \eqref{div3}, we know that $\frac{L(C_{2p},1)}{3\Omega_1}$ and $\frac{L(C_{2^2p^2},1)}{3\Omega_1}$ lie in $\ov{\BZ}_3$. Recall also that we have $\sqrt[3]{2p}\equiv 1\mod 3^{\frac{1}{3}}$ and that the period relations $\Omega_{2p}=\frac{\Omega_1}{\sqrt[3]{2p}}$ and $\Omega_{2^2p^2}=\frac{\Omega_1}{\sqrt[3]{2^2p^2}}$. Then we know that
\[\frac{L(C_{2p},1)}{3\Omega_1}-\frac{L(C_{2p},1)}{3\Omega_{2p}} \quad \text{ and }\quad \frac{L(C_{2^2p^2},1)}{3\Omega_1}-\frac{L(C_{2^2p^2},1)}{3\Omega_{2^2p^2}}\]
have positive $3$-adic valuations. Now, since $\frac{L(C_{2p},1)}{\Omega_{2p}}, \frac{L(C_{2^2p^2},1)}{\Omega_{2^2p^2}}\in \BQ$, from \eqref{div3} and \eqref{s11}, we obtain
$$
  \frac{L(C_{2p},1)}{3\Omega_{2p}}\equiv \frac{L(C_{2^2p^2},1)}{3\Omega_{2^2p^2}}\mod 3.
$$
This congruence combined with the second assertion of (1) in Theorem \ref{sum-exp-3bsd} shows that
$$
  \frac{L(C_{2p},1)}{3\Omega_{2p}}\equiv \frac{L(C_{2^2p^2},1)}{3\Omega_{2^2p^2}}\equiv1\mod 3.
$$
This completes the proof of Theorem \ref{mod3-bsd-2p}.

\medskip

Next, we give a $3$-descent result for $C_{2p}, C_{2^2p^2}$ (resp.\ $C_{2p^2}, C_{2^2p}$) when $p\equiv 5\mod9$ (resp.\ $p\equiv2\mod 9$).

\begin{prop}\label{3-descent} Let $C_{2^i p^j}\colon x^2+y^2=2^i p^j$, where $i,j\in \{1,2\}$ and $p\equiv 2,5\bmod 9$. Assume in addition that $2^i p^j\equiv \pm 1\bmod 9$. Then the $3$-part of the Tate--Shafarevich group $\Sha(C_{2^ip^j})[3]$ of $C_{2^ip^j}$ over $\BQ$ is trivial.
\end{prop}

\begin{proof}
We write
$$E_n\colon y^2z=x^3-2^43^3n^2z^3.$$
This is birationally equivalent to $C_n$, and over $\BQ$ it is $3$-isogenous to
$$E_n'\colon y^2z=x^3+2^4n^2z^3.$$
Explicitly, the isogeny $\phi\colon E_n \to E_n'$ is given by
$$
\phi(x,y,z)=(3^{-2}(x^4-2^63^3n^2xz^3), 3^{-3}y(x^3+2^73^3n^2z^3),x^3z)
$$
and the kernel of $\phi$ is $\{(0,1,0), (0, \pm 2^2(\sqrt{-3})^3 n, 1)\}$.

By assumption, we have $n=2^ip^j\equiv \pm 1\bmod 9$.
This happens precisely when $n=2p, 2^2p^2$ and $p\equiv 5\bmod 9$ or $n=2p^2, 2^2p$ and $p\equiv 2\bmod 9$.
Since $n$ has no prime divisor congruent to $1$ modulo $3$, hypothesis (H) of \cite[Th\'{e}or\`{e}me 2.9]{satge2} clearly holds, and thus (1) of the theorem gives us
\begin{equation}\label{selphi}\Sel_\phi(E_n)=\BZ/3\BZ\quad \text{ and }\quad \Sel_{\hat{\phi}}(E_n')=0,
\end{equation}
where $\hat{\phi}$ denotes the dual isogeny of $\phi$. Here, $\Sel_f(E)$ denotes the Selmer group of an elliptic curve $E$ over $\BQ$ associated to an isogeny $f$ from $E$ to another elliptic curve $E'$. Recall that we have an exact sequence (see, for example, \cite{sil})
$$0\to E_n'(\BQ)/\phi(E_n(\BQ)) \to \Sel_\phi(E_n)\to \Sha(E_n)[\phi]\to 0.$$
where $\Sha(E_n)[\phi]$ denotes the subgroup of elements in $\Sha(E_n)$ killed by $\phi$.
Now, we know that $E_n(\BQ)_{\mathrm{tor}}=0$, $E_n'(\BQ)_{\mathrm{tor}}=\{(0,1,0),(0,\pm 2^2D,1)\}\simeq \BZ/3\BZ$, thus $\Sha(E_n)[\phi]=0$. We also have the following exact sequence (see, for example, \cite{HYS})
$$0\to \Sha(E_n)[\phi]\to \Sha(E_n)[3]\xrightarrow{\phi} \Sha(E_n')[\hat{\phi}]=0,$$
where the last equality is immediate from \eqref{selphi}. Thus we obtain $\Sha(E_n)[3]$ is trivial, as required.
\end{proof}

We end this section by completing the proof the $3$-part of the Birch--Swinnerton-Dyer conjecture of the curves appearing in Theorem \ref{main1} and Corollary \ref{main1+}.

\begin{proof}[Proof of Theorem \ref{main1} and Corollary \ref{main1+}.]
We will just give the details for the curves $C_{2p}, C_{2p^2}$ under the assumption of $p\equiv 5\mod 9$, and the other case is similar. By a theorem of Satg\'e \cite{satge} and the theorem of Gross--Zagier and Kolyvagin, we know that $L(C_{2p^2},s)$ has a simple zero at $s=1$, and $C_{2p^2}(\BQ)$ is of rank one. By Theorem \ref{mod3-bsd-2p}, we know that $L(C_{2p},1)$ does not vanish, so it follows by the theorem of Gross--Zagier and Kolyvagin or Coates--Wiles \cite{CW} that $C_{2p}(\BQ)$ is finite. We also remark that in these cases, the Tate--Shafarevich group of $C_{2p}$ and $C_{2p^2}$ are finite.

From \cite{CST}, we know that the $3$-part Birch--Swinnerton-Dyer conjecture holds for the product curve $C_{2p}\times C_{2p^2}$. Thus
\begin{equation}\label{bsd-prod}
\begin{aligned}
  &\ord_3\left(\frac{L(C_{2p},1)L'(C_{2p^2},1)}{\Omega_{2p}\Omega_{2p^2}\cdot\wh{h}(P)}\right)\\
  &=\ord_3\left(T(C_{2p})T(C_{2p^2})\cdot\#(\Sha(C_{2p}))\#(\Sha(C_{2p^2}))\right)
\end{aligned}
\end{equation}
where $T(C_{2p})=\prod_{\ell\mid 6p}m_{\ell}(C_{2p})$ and $T(C_{2p^2})=\prod_{\ell\mid 6p}m_\ell(C_{2p^2})$ are the products of local Tamagawa numbers of $C_{2p}$ and $C_{2p^2}$, respectively. Here, $m_\ell(C_n)$ denotes the local Tamagawa number of $C_n$ at the prime $\ell$. Since $C_{2p^2}(\BQ)$ is of rank one, we let $P$ be a generator of the free part of $C_{2p^2}(\BQ)$, and denote by $\wh{h}(P)$ the N\'{e}ron--Tate height of $P$. Note that the local Tamagawa numbers of $C_{2p}$ and $C_{2p^2}$ are all equal to $1$ except that $\ell=3$, where $m_3(C_{2p})=3$ (see \cite{zk87}). By Theorem \ref{mod3-bsd-2p} and Proposition \ref{3-descent},  the $3$-part of the Birch--Swinnerton-Dyer conjecture holds for the curve $C_{2p}$. Thus
\begin{equation}\label{bsd-2p2}
  \ord_3\left(\frac{L(C_{2p},1)}{\Omega_{2p}}\right)
  =\ord_3\left(T(C_{2p})\cdot\#(\Sha(C_{2p}))\right).
\end{equation}
Now the $3$-part of the Birch--Swinnerton-Dyer conjecture for $C_{2p^2}$ follows from \eqref{bsd-prod} and \eqref{bsd-2p2}, and Corollary \ref{main1+} follows as a consequence.
\end{proof}

\section{Some relations between the $2$-Selmer group of $C_{2p^j}$ and the $2$-class group of $\BQ(\sqrt[3]{p})$}\label{section4}

The aim of this section is to give a proof of Theorem \ref{main2}. Given $j\in \{1,2\}$, the curve $C_{2p^j}$ is a twist of $C_2$ by the cubic extension $\BQ(\sqrt[3]{p^j})/\BQ$ of discriminant prime to $2$. Since a global minimal Weierstrass model of $C_2$ is given by $y^2=x^3-27$, we know that
\begin{equation}\label{min-model}y^2=x^3-27p^{2j}\qquad (j=1,2)
\end{equation}
gives a minimal model of $C_{2p^j}$ at $2$.

In the remainder of this section, we write $E$ for $C_{2p^j}$ for simplicity, and we work with the equation given in \eqref{min-model}. We will study the relation between the $2$-class group of $L=\BQ(\sqrt[3]{p})$ and the $2$-Selmer group $\Sel_2(E)$.
It is well known that the torsion part $E(\BQ)_{\mathrm{tor}}$ of $E(\BQ)$ is trivial, and we identify
\[\frac{\BQ[x]}{(x^3-27p^2)}=\frac{\BQ[x]}{(x^3-p^2)}=\frac{\BQ[x]}{(x^3-p)}=\frac{\BQ[x]}{(x^3-27p^4)}\]
with $L$. To see where these fields come from, one may look at the classical proof of the weak Mordell--Weil theorem in the case $E(\BQ)[2]=0$ (see, for example, \cite{cassels, sil}). The proof of the following lemma can be found in \cite{cl}.

\begin{lem}\label{2dec1}
Let $\mathrm{Res}_{L/\BQ}\mu_2$ be the restriction of scalars of $\mu_2$ from $L$ to $\BQ$, treated as a $G_\BQ$-module. Then we have the exact sequence
\[0\ra E[2]\ra \mathrm{Res}_{L/\BQ}\mu_2\stackrel{\BN}{\longrightarrow}\mu_2\ra 0,\]
where all of the terms in the exact sequence are considered as $\ov{\BQ}$-points, and $\BN$ denotes the multiplication of the three components of $\mathrm{Res}_{L/\BQ}\mu_2$. Moreover, we have an isomorphism
\begin{equation}\label{h2g}
  H^1(\BQ, E[2])\simeq (L^\times/(L^\times)^2)_\Box
\end{equation}
where
\[(L^\times/(L^\times)^2)_\Box:=\{\alpha\in L^\times/(L^\times)^2: N(\alpha)\in (\BQ^\times)^2\}\]
and $N$ is the norm map from $L$ to $\BQ$.
\end{lem}

We denote by $\delta$ the morphism which is the composition of the Kummer map
\[E(\BQ)/2E(\BQ)\ra H^1(\BQ, E[2])\]
with the isomorphism in \eqref{h2g}, which we still call the global Kummer map. Then we have the following description of $\delta$ (for a detailed proof, see \cite{cl}):

\begin{lem}\label{2dec2}
The map $\delta$ is given by
\[\delta\colon E(\BQ)/2E(\BQ)\ra (L^\times/(L^\times)^2)_\Box,\qquad P\mapsto x(P)-\beta,\]
where $x(P)$ is the $x$-coordinate of $P$ and $\beta$ is the image of $x$ in $L=\BQ[x]/(x^3-27p^{2j})$.
\end{lem}

We remark that the above description of the map $\delta$ dates back to the proof of the Mordell--Weil theorem when $E(\BQ)[2]=0$. See the book of Cassels \cite{cassels} for more details. Similarly, we can state Lemmas \ref{2dec1} and \ref{2dec2} over $\BQ_q$ for all primes $q$, finite or infinite. We put $L_q=L\otimes_\BQ\BQ_q$. Then we obtain the local Kummer map $\delta_q$, which is the composition of the Kummer map
\[E(\BQ_q)/2E(\BQ_q)\ra H^1(\BQ_q, E[2])\]
with the isomorphism
\[H^1(\BQ_q, E[2])\simeq (L^\times_q/(L^\times_q)^2)_\Box,\]
where
\[(L^\times_q/(L^\times_q)^2)_\Box:=\{x\in L^\times_q/(L^\times_q)^2: N(x)\in (\BQ^\times_q)^2\}\]
and $N$ is the norm map from $L_q$ to $\BQ_q$.

\begin{lem}\label{2dec3}
The map $\delta_q$ is given by
\[\delta_q\colon E(\BQ_q)/2E(\BQ_q)\ra (L^\times_q/(L^\times_q)^2)_\Box,\qquad P\mapsto x(P)-\beta,\]
where $\beta$ is the image of $x$ in $L_q=\BQ_q[x]/(x^3-27p^{2j})$.
\end{lem}

To obtain a precise description of the Selmer group $\Sel_2(E)$, we need to describe the image of all $\delta_q$. We denote by $\fO$ the ring of integers $\BZ[x]/(x^3-27p^{2j})$ of $L$, and we set $\fO_q=\fO\otimes_\BZ\BZ_q$.

\begin{lem}\label{square}
For any prime $q$ and for any $P\in E(\BQ_q)$, we have
\[\delta_q(P)\in (\fO^\times_q/(\fO^\times_q)^2)_\Box\]
where $(\fO^\times_q/(\fO^\times_q)^2)_\Box$ again denotes a subset of $\fO^\times_q/(\fO^\times_q)^2$ consisting of square norms.
\end{lem}

\begin{proof}

For $q\neq 2$, this is Lemma 6 of \cite{GP}. Here, we note that since the Kodaira type of $E$ at each prime $\ell$ is IV or IV*, $\bold{\Phi}[2]$ is trivial, where $\bold{\Phi}=\widetilde{\mathcal{E}}(\overline{\mathbb{F}}_\ell)/\widetilde{\mathcal{E}}^0(\overline{\mathbb{F}}_\ell)$ denotes the component group scheme of the \Neron model $\mathcal{E}$ of $E$ over $\BQ_\ell$, and we thank the referee for pointing this out. Here, $\widetilde{\phantom{a}}$ denotes reduction modulo~$\ell$ and $\widetilde{\mathcal{E}}^0$ denotes the identity component. For $q=2$, we just need to show $\ord_2(x(P)-\beta)$ is even, where $\ord_2$ denotes the normalized $2$-adic valuation with $\ord_2(2)=1$. Note that
\begin{equation}\label{x3}
  x^3-27p^{2j}\equiv (x+1)(x^2+x+1) \mod 2.
\end{equation}
We write also $3\sqrt[3]{p^{2j}}$ for the unique root in $\BZ_2$ of $x^3-27p^{2j}=0$ satisfying $3\sqrt[3]{p^{2j}}\equiv 1\bmod 2$ obtained by Hensel's lemma. Thus we have a factorisation
$$x^3-27p^{2j}=(x-3\sqrt[3]{p^{2j}})(x^2+3\sqrt[3]{p^{2j}}x+(3\sqrt[3]{p^{2j}})^2)$$
over $\BZ_2$. Now, we have
$$2\ord_2(y)=\ord_2(x-3\sqrt[3]{p^{2j}})+\ord_2(x^2+3\sqrt[3]{p^{2j}}x+(3\sqrt[3]{p^{2j}})^2)$$
for $x=x(P)$ and $y=y(P)$ in $\BQ_2$, so we know that the right hand side is even. If $\ord_2(x)<0$, then $\ord_2(x-3\sqrt[3]{p^{2j}})=\ord_2(x)$ and $\ord_2(x^2+3\sqrt[3]{p^{2j}}x+(3\sqrt[3]{p^{2j}})^2)=2\ord_2(x)$, so it follows that $3\ord_3(x)$ is even, and thus $\ord_2(x)$ is even. If $\ord_2(x)\geq 0$, noting that $x^2+x+1\nequiv 0 \bmod 2$ in \eqref{x3}, we see that $2\ord_2(y)=\ord_2(x-3\sqrt[3]{p^{2j}})$, and thus  $\ord_2(x-3\sqrt[3]{p^{2j}})$ is even, as required.
\end{proof}

Now, we can describe the image of $\delta_q$.

\begin{lem}\label{loc-k}
\begin{enumerate}
\noindent  \item The image of $\delta_\infty$ is trivial.
  \item If $q\neq 2, \infty$, we have
  \[\mathrm{im}(\delta_q)=(\fO^\times_q/(\fO^\times_q)^2)_\Box.\]
  \item The image of $\delta_2$ is a subgroup of $(\fO^\times_2/(\fO^\times_2)^2)_\Box$ of index $2$, and contains the subgroup of elements in $(\fO^\times_2/(\fO^\times_2)^2)_\Box$ which are congruent to $1$ modulo $4$.
\end{enumerate}

\end{lem}

\begin{proof}

Note that $L_\infty=\BC\times \BR$ and we know that $(L^\times_\infty/(L^\times_\infty)^2)_\Box=1$, thus $\mathrm{im}(\delta_\infty)$ is trivial and (1) follows.

For (2) and (3), we know that the norm of $L^\times_q$ in $\BQ^\times_q$ has index $3$ by local class field theory. Thus, in the commutative diagram
\[\xymatrix{
  0  \ar[r]^{} & (\fO_q^\times)^2 \ar[d]_{} \ar[r]^{} & \fO^\times_q \ar[d]_{N} \ar[r]^{} & \fO^\times_q/(\fO^\times_q)^2 \ar[d]_{\lambda} \ar[r]^{} & 0 \\
  0 \ar[r]^{} & (\BZ^\times_q)^2 \ar[r]^{} & \BZ^\times_q \ar[r]^{} & \BZ^\times_q/(\BZ^\times_q)^2 \ar[r]^{} & 0   }\]
where $N$ is the restriction of the norm map from $L_q$ to $\BQ_q$, we see that $\lambda$ is surjective. Furthermore, from the exact sequence
\[0\ra (\fO^\times_q/(\fO^\times_q)^2)_\Box\ra \fO^\times_q/(\fO^\times_q)^2\stackrel{\lambda}{\longrightarrow}\BZ^\times_q/(\BZ^\times_q)^2\ra 0,\]
we obtain
\[\#(\fO^\times_q/(\fO^\times_q)^2)_\Box=\frac{\#(\fO^\times_q/(\fO^\times_q)^2)}{\#(\BZ^\times_q/(\BZ^\times_q)^2)}.\]
Note that $\#(\BZ^\times_q/(\BZ^\times_q)^2)=2$ or $2^2$ according as $q \neq 2$ or $q=2$. Suppose $x^3-27p^{2j}$ is a product of $\ell$ irreducible polynomials modulo $q$. For $q\neq 2$, we have $\#(\fO^\times_q/(\fO^\times_q)^2)=2^\ell$. For $q=2$, note that we have $\fO_2^\times=\BZ_2^\times \times \BZ_2[\gamma]^\times$ where $\gamma$ is a root of the irreducible polynomial $x^2+3\sqrt[3]{p^{2j}}x+(3\sqrt[3]{p^{2j}})^2\in \BZ_2[x]$ from the proof of Lemma \ref{square}. Note also that $(1+2\BZ_2[\gamma])^2$ is a subgroup of $1+4\BZ_2[\gamma]$ of index $2$. Thus we have $\#(\fO^\times_2/(\fO^\times_2)^2)=2^{2+3}=2^{5}$. It follows that $\#(\fO^\times_q/(\fO^\times_q)^2)_\Box=2^{\ell-1}$ or $2^3$ according as $q\neq 2$ or $q=2$. On the other hand, a direct computation shows that
\[\#(E(\BQ_q)[2])=2^{\ell-1}.\]
Furthermore, since $E(\BQ_q)$ has a subgroup isomorphic to $\BZ_q$ of finite index, we have
$\#(E(\BQ_q)/2E(\BQ_q))=\#(E(\BQ_q)[2])=2^{\ell-1}$ or $2\cdot \#E(\BQ_q)[2]=2^2$ according as $q \neq 2$ or $q=2$.
Therefore (2) follows, and it remains to show that $\mathrm{im}(\delta_2)$ contains the subgroup $(1+4\fO_2/(\fO^\times_2)^2)_\Box=(1+4\BZ_2[\gamma]/(\BZ_2[\gamma]^\times)^2)$ of order $2$.

Let $\wh{E}$ be the formal group of $E$ over $\BZ_2$ given by \eqref{min-model}. Let $P\in \wh{E}(2\BZ_2)$, and write $t=-\frac{x}{y}$ for the uniformizer of $\wh{E}$. Then we have
\[x(P)=t^{-2}+O(t^2)\]
since $a_1=a_2=a_3=0$. Then
$$
  \delta_2(P)=x(P)-\beta=t^{-2}-\beta+O(t^2)\equiv 1-\beta t^2\mod (L^\times_2)^2,
$$
where the last equality is obtained by noting that $t\in 2\BZ_2$ and that the units in $\fO_2$ congruent to $1$ modulo $16$ lie in $(\fO_2^\times)^2$. Now, take $P\in \widehat{E}(2\BZ_2)$ corresponding to $t=2u\in 2\BZ_2\backslash 4\BZ_2$.  Then $\delta_2(P)\equiv 1-4\gamma u^2$ clearly lies in the non-trivial class of $(1+4\BZ_2[\gamma]/(\BZ_2[\gamma]^\times)^2)$. This concludes the proof of  Lemma \ref{loc-k}.
\end{proof}

We can now state a relation between the $2$-Selmer group $\Sel_2(E)$ of $E$ and the $2$-class group of $L=\BQ(\sqrt[3]p)$. Write $\epsilon(E/\BQ)$ for the global root number of $E/\BQ$. Recall from the introduction that  $\rank_2\left(Cl(F)\right):=\dim_{\BF_2}(Cl(F)/2 Cl(F))$ denotes the $2$-rank of $Cl(F)$.

\begin{prop}\label{rel2}
Let $k=\mathrm{rank}_2(Cl(L))$. Then
$$
\mathrm{dim}_{\BF_2}\Sel_2(E)=
 \left\{ \begin{array}{ll}
k&\mbox{if $\epsilon(E/\BQ)=(-1)^k$,}\\
k+1 & \mbox{if $\epsilon(E/\BQ)=(-1)^{k+1}$.}\\
       \end{array} \right.
$$
\end{prop}

\begin{proof}
Since we may follow the methods in \cite{cl}, we will only state the key ideas in the argument. We define two subgroups
\[N_1=\{\alpha\in L^\times/(L^\times)^2: L(\sqrt{\alpha})/L \textrm{ is unramified}\}\]
and
\[N_2=\{\alpha\in L^\times/(L^\times)^2: \alpha>0, (\alpha)=I^2, I\subset L \textrm{ is a fractional ideal}\}\]
of $(L^\times/(L^\times)^2)_\Box$. Then from the description of the local Kummer map given in \eqref{loc-k}, we see that the $2$-Selmer group
$\Sel_2(E)$ satisfies
\[N_1\subseteq \Sel_2(E)\subseteq N_2. \]
Furthermore, by Kummer theory, we have
\[N_1\simeq \Hom(\Gal(M/L),\mu_2)\]
where $M$ is the maximal abelian extension of exponent two unramified everywhere, so that $\Gal(M/L)=Cl(L)[2]$. Note that $\#(Cl(L)[2])=\#(Cl(L)/2Cl(L))$, since $Cl(L)$ is a finite group. Thus we have
\[\#(N_1)=2^k \text{ where }\dim_{\BF_2} Cl(L)[2]=k.\]
By Dirichlet's unit theorem, the free part of the group $\CO^\times_L$ of units in $L$ has rank one. Let $u_L$ be a generator. By choosing $-u_L$ or $u_L$,
we may assume $u_L>0$, and by choosing $u_L$ or $u_L^{-1}$, we may assume $u_L>1$. Now, we define a map
\[N_2\ra Cl(L)[2]\quad \alpha\mapsto [I]\]
where $I$ is a fractional ideal of $L$ such that $I^2=(\alpha)$ and $[I]$ denotes its ideal class in $Cl(L)$. This map
is surjective, and the kernel of this map is given by $u^\BZ_L/u^{2\BZ}_L\simeq \BZ/2\BZ$. Hence we have
\[\dim_{\BF_2}N_2=k+1.\]
From all of the above considerations and the $2$-parity conjecture proved in \cite{monsky}, we obtain the assertion of Proposition \ref{rel2}. This also concludes the proof of Theorem \ref{main2}.
\end{proof}

\begin{appendix}

\section{Numerical examples}\label{appendixB}
Let $p\equiv 2,5\bmod 9$ be an odd prime number. Let $L=\BQ(\sqrt[3]{p})$, and write $Cl(L)$ for the ideal class group of $L$. We ran a numerical test for primes $p<1000000$, and obtained that $1629$ out of $13068$ primes $p\equiv 5\bmod 9$, and $1852$ out of $13099$ primes $p\equiv 2 \bmod 9$ satisfy condition
\begin{equation}\label{nontriv}\rank_2(Cl(L))=\dim_{\BF_2}\left(Cl(L)/2Cl(L)\right)\geq 2
\end{equation}
from Theorem \ref{main2}.
The following lists are small hand-picked samples of primes $p$ such that $Cl(L)$ satisfies \eqref{nontriv}, including all the cases when the Tate--Shafarevich group of $C_{2p}$ (resp.\ $C_{2p^2}$) is strictly bigger than $\BZ/2\BZ \oplus \BZ/2\BZ$ for $p\equiv 5\mod 9$ (resp.\ $p\equiv 2\bmod 9$) and $p<1000000$. All the numerical results were obtained using the computer algebra packages \emph{Magma} \cite{Magma} and \emph{PARI/GP} \cite{PARI2}.

\begin{footnotesize}

\begin{equation}\nonumber
\begin{array}{lll}
p\equiv 5 \bmod 9 &Cl(L)&\Sha(C_{2p})[2]\\
\hline
 113 &  \BZ/2\BZ \oplus\BZ/2\BZ & \BZ/2\BZ \oplus\BZ/2\BZ \\
 3209 & \BZ/2\BZ \oplus\BZ/68\BZ & \BZ/2\BZ \oplus\BZ/2\BZ \\
 9941 & \BZ/2\BZ \oplus\BZ/4\BZ & \BZ/2\BZ \oplus\BZ/2\BZ \\
15053 & \BZ/2\BZ \oplus\BZ/14\BZ& \BZ/2\BZ \oplus\BZ/2\BZ \\
17573 & \BZ/2\BZ \oplus\BZ/374\BZ & \BZ/2\BZ \oplus\BZ/2\BZ \\
18257 &  \BZ/2\BZ \oplus\BZ/170\BZ & \BZ/2\BZ \oplus\BZ/2\BZ \\
 24197 &  \BZ/2\BZ \oplus\BZ/20\BZ & \BZ/2\BZ \oplus\BZ/2\BZ \\
32009 & \BZ/2\BZ \oplus\BZ/16\BZ & \BZ/2\BZ \oplus\BZ/2\BZ\\
35969 & \BZ/2\BZ \oplus\BZ/140\BZ& \BZ/2\BZ \oplus\BZ/2\BZ \\
40577 & \BZ/4\BZ \oplus\BZ/28\BZ & \BZ/2\BZ \oplus\BZ/2\BZ\\
41981 & \BZ/2\BZ \oplus\BZ/26\BZ & \BZ/2\BZ \oplus\BZ/2\BZ\\
46229 & \BZ/2\BZ \oplus \BZ/2\BZ \oplus \BZ/4\BZ & \BZ/2\BZ \oplus\BZ/2\BZ\\
54869 & \BZ/2\BZ \oplus\BZ/1190\BZ & \BZ/2\BZ \oplus\BZ/2\BZ\\
61169 &  \BZ/2\BZ \oplus\BZ/50\BZ & \BZ/2\BZ \oplus\BZ/2\BZ \\
61547 & \BZ/2\BZ \oplus\BZ/314\BZ & \BZ/2\BZ \oplus\BZ/2\BZ \\
73013 & \BZ/2\BZ \oplus \BZ/2\BZ \oplus \BZ/2\BZ & \BZ/2\BZ \oplus\BZ/2\BZ\\
 81077 & \BZ/2\BZ \oplus \BZ/2\BZ \oplus \BZ/2\BZ \oplus\BZ/2\BZ & \BZ/2\BZ \oplus \BZ/2\BZ \oplus \BZ/2\BZ \oplus\BZ/2\BZ \\
  97943 & \BZ/2\BZ \oplus\BZ/10\BZ& \BZ/2\BZ \oplus\BZ/2\BZ \\
 166667 & \BZ/2\BZ \oplus \BZ/2\BZ \oplus \BZ/14\BZ & \BZ/2\BZ \oplus\BZ/2\BZ\\
169007 & \BZ/4\BZ \oplus\BZ/124\BZ & \BZ/2\BZ \oplus\BZ/2\BZ\\
195581 & \BZ/4\BZ \oplus\BZ/184\BZ  & \BZ/2\BZ \oplus\BZ/2\BZ\\
206489 & \BZ/4\BZ \oplus\BZ/20\BZ & \BZ/2\BZ \oplus\BZ/2\BZ\\
483017 &  \BZ/2\BZ \oplus \BZ/2\BZ \oplus\BZ/2\BZ \oplus\BZ/4\BZ  & \BZ/2\BZ \oplus \BZ/2\BZ \oplus\BZ/2\BZ \oplus\BZ/2\BZ\\
635909 &  \BZ/2\BZ \oplus \BZ/2\BZ \oplus\BZ/2\BZ \oplus\BZ/2\BZ &  \BZ/2\BZ \oplus \BZ/2\BZ \oplus\BZ/2\BZ \oplus\BZ/2\BZ \\
805073 & \BZ/2\BZ \oplus \BZ/2\BZ \oplus\BZ/2\BZ \oplus\BZ/2\BZ &  \BZ/2\BZ \oplus \BZ/2\BZ \oplus\BZ/2\BZ \oplus\BZ/2\BZ \\
964589 & \BZ/2\BZ \oplus \BZ/2\BZ \oplus\BZ/2\BZ \oplus\BZ/2\BZ  &  \BZ/2\BZ \oplus \BZ/2\BZ \oplus\BZ/2\BZ \oplus\BZ/2\BZ \\

 \hline
\end{array}
\end{equation}

\begin{equation}\nonumber
\begin{array}{lll}
p\equiv 2 \bmod 9 &Cl(L)&\Sha(C_{2p^2})[2]\\
\hline
443 & \BZ/2\BZ \oplus\BZ/2\BZ  & \BZ/2\BZ \oplus\BZ/2\BZ \\
 857 & \BZ/2\BZ \oplus\BZ/28\BZ & \BZ/2\BZ \oplus\BZ/2\BZ \\
 4799 &  \BZ/2\BZ \oplus \BZ/2\BZ \oplus \BZ/20\BZ & \BZ/2\BZ \oplus\BZ/2\BZ\\
 5987 & \BZ/2\BZ \oplus\BZ/64\BZ & \BZ/2\BZ \oplus\BZ/2\BZ \\
 9011 & \BZ/2\BZ \oplus\BZ/4\BZ & \BZ/2\BZ \oplus\BZ/2\BZ \\
 9749 & \BZ/4\BZ \oplus\BZ/16\BZ & \BZ/2\BZ \oplus\BZ/2\BZ \\
 13043 & \BZ/2\BZ \oplus\BZ/8\BZ & \BZ/2\BZ \oplus\BZ/2\BZ \\
 17579 & \BZ/2\BZ \oplus \BZ/2\BZ \oplus \BZ/230\BZ & \BZ/2\BZ \oplus\BZ/2\BZ \\
 26111 & \BZ/2\BZ \oplus\BZ/52\BZ & \BZ/2\BZ \oplus\BZ/2\BZ \\
 31547 &  \BZ/2\BZ \oplus\BZ/50\BZ & \BZ/2\BZ \oplus\BZ/2\BZ \\
 36263 & \BZ/2\BZ \oplus\BZ/70\BZ & \BZ/2\BZ \oplus\BZ/2\BZ \\
  47387 &  \BZ/2\BZ \oplus \BZ/2\BZ \oplus \BZ/2\BZ & \BZ/2\BZ \oplus\BZ/2\BZ \\
 58727 &  \BZ/2\BZ \oplus\BZ/32\BZ & \BZ/2\BZ \oplus\BZ/2\BZ \\
 60149 & \BZ/2\BZ \oplus\BZ/74\BZ & \BZ/2\BZ \oplus\BZ/2\BZ \\
65063 & \BZ/2\BZ \oplus \BZ/4\BZ \oplus\BZ/8\BZ  & \BZ/2\BZ \oplus\BZ/2\BZ \\
 78437 & \BZ/4\BZ \oplus\BZ/4\BZ & \BZ/2\BZ \oplus\BZ/2\BZ \\
 79967 &  \BZ/2\BZ \oplus\BZ/40\BZ & \BZ/2\BZ \oplus\BZ/2\BZ \\
 96329 &  \BZ/2\BZ \oplus\BZ/22\BZ & \BZ/2\BZ \oplus\BZ/2\BZ \\
 99371 & \BZ/2\BZ \oplus \BZ/2\BZ \oplus \BZ/4\BZ & \BZ/2\BZ \oplus\BZ/2\BZ \\
  125003 & \BZ/2\BZ \oplus \BZ/4\BZ \oplus\BZ/700\BZ & \BZ/2\BZ \oplus\BZ/2\BZ\\
 167087 & \BZ/2\BZ \oplus \BZ/2\BZ \oplus \BZ/656\BZ & \BZ/2\BZ \oplus\BZ/2\BZ\\
 266663 & \BZ/2\BZ \oplus \BZ/2\BZ \oplus \BZ/2\BZ \oplus\BZ/4\BZ & \BZ/2\BZ \oplus \BZ/2\BZ \oplus \BZ/2\BZ \oplus\BZ/2\BZ \\
 402131 & \BZ/2\BZ \oplus \BZ/2\BZ \oplus\BZ/2\BZ \oplus\BZ/2\BZ & \BZ/2\BZ \oplus \BZ/2\BZ \oplus\BZ/2\BZ \oplus\BZ/2\BZ \\
424163 &  \BZ/2\BZ \oplus \BZ/2\BZ \oplus\BZ/2\BZ \oplus\BZ/4\BZ  & \BZ/2\BZ \oplus \BZ/2\BZ \oplus\BZ/2\BZ \oplus\BZ/2\BZ \\
  521831 & \BZ/2\BZ \oplus \BZ/2\BZ \oplus\BZ/2\BZ \oplus\BZ/2\BZ & \BZ/2\BZ \oplus \BZ/2\BZ \oplus\BZ/2\BZ \oplus\BZ/2\BZ \\
 916103 & \BZ/2\BZ \oplus \BZ/2\BZ \oplus\BZ/2\BZ \oplus\BZ/2\BZ  & \BZ/2\BZ \oplus \BZ/2\BZ \oplus\BZ/2\BZ \oplus\BZ/2\BZ\\
  \hline
\end{array}
\end{equation}

\end{footnotesize}

\end{appendix}

\subsection*{Acknowledgement}
The authors would like to thank John Coates for his comments on this work. The second author would like to thank Ye Tian for the suggestion given during his PhD to study the Birch--Swinnerton-Dyer conjecture for the individual curves in the product of curves. The authors are very grateful to the referee whose valuable comments lead to an improvement of the quality of this paper.

The first author gratefully acknowledges the support of the SFB 1085 ``Higher invariants'' at the University of Regensburg and would also like to thank the Max Planck Institute for Mathematics in Bonn for its support and hospitality. The second author is supported by NSFC-11901332.

\vspace{20pt}

\begin{center}
\begin{tabular}{@{}p{2.5in}p{2.1in}}
Yukako Kezuka & Yongxiong Li\\
Max-Planck-Institut f\"{u}r Mathematik & Yau Mathematical Sciences Center\\
Bonn, Germany & Tsinghua University\\
\emph{ykezuka@mpim-bonn.mpg.de}& Beijing, China\\
& \emph {\it liyx\_1029@tsinghua.edu.cn} \\
\end{tabular}
\end{center}

\begin{thebibliography}{10}

 \bibitem{Magma}
 W.\ Bosma, J.\ Cannon, C.\ Playoust, {\em The Magma algebra system. I.} The user language, J.\ Symbolic Comput.\  24 (1997), 235--265.

\bibitem{CST} L.\ Cai, J.\ Shu, Y.\ Tian,  {\em Cube sum problem and an explicit Gross--Zagier formula}, Amer.\ J.\ Math.\ 139 (2017), no.\ 3, 785--816.

\bibitem{cassels62} J.\ Cassels, {\em Arithmetic on curves of genus 1. IV. Proof of the Hauptvermutung},   J. Reine Angew.\ Math.\ 211 (1962), 95--112.


\bibitem{cassels} J.\ Cassels, {\em Lectures on elliptic curves},  London Mathematical Society Student Texts, 24. Cambridge University Press, Cambridge, 1991. vi+137 pp. ISBN 0-521-41517-9; 0-521-42530-1.

\bibitem{CW} J.\ Coates, A.\ Wiles, {\em On the conjecture of Birch and Swinnerton-Dyer}, Invent.\ Math.\ 39 (1977), 223--251.


\bibitem{coates1} J.\ Coates,  {\em Lectures on the Birch--Swinnerton-Dyer conjecture}, ICCM Not.\ 1 (2013), no.\ 2, 29--46.

\bibitem{CL} J.\ Coates, Y.\ Li, {\em Non-vanishing theorems for central $L$-values of some elliptic curves with complex multiplication}, Proc.\ Lond.\ Math.\ Soc.\ (3) 121 (2020), no. 6, 1531--1578.



\bibitem{GP}B.\ Gross, J.\ Parson, {\em On the local divisibility of Heegner points},  Number Theory, Analysis and Geometry, 215--241, Springer, New York, 2012.

\bibitem{HYS}Y.\ Hu, H.\ Yin, J.\ Shu, {\em   An explicit Gross-Zagier formula related to the Sylvester conjecture},   Trans.\ Amer.\ Math.\ Soc.\ 372 (2019), no.\ 10, 6905--6925.


\bibitem{yuka}Y.\ Kezuka, {\em   On the $p$-part of the Birch--Swinnerton-Dyer conjecture for elliptic curves with complex multiplication by the ring of integers of $\Bbb Q(\sqrt{-3})$},  Math. Proc. Cambridge Philos. Soc. 164 (2018), no. 1, 67--98.


\bibitem{cl}C.\ Li, {\em   $2$-Selmer groups, $2$-class groups and rational points on elliptic curves},   Trans.\ Amer.\ Math.\ Soc.\ 371 (2019), no.\ 7, 4631--4653.


\bibitem{liv95}E.\ Liverance, {\em A formula for the root number of a family of elliptic curves},  J.\ Number Theory 51 (1995), no.\ 2, 288--305.

\bibitem{monsky}P.\ Monsky, {\em Generalizing the Birch-Stephens theorem. I.\ Modular curves},   Math.\ Z.\ 221 (1996), no.\ 3, 415--420.



 \bibitem{PARI2}
    The PARI~Group, PARI/GP version {\tt 2.11.1}, Univ.\ Bordeaux, 2018,
    \url{http://pari.math.u-bordeaux.fr/}.

\bibitem{qiu}D.\ Qiu, X.\ Zhang, {\em    Elliptic curves with CM by $\sqrt{-3}$ and $3$-adic valuations of their $L$-series},   Manuscripta Math.\ 108 (2002), no.\ 3, 385--397.


\bibitem{razar1}M.\ Razar, {\em   The non-vanishing of $L(1)$ for certain elliptic curves with no first descents},   Amer.\ J.\ Math.\ 96 (1974), 104--126.

\bibitem{razar2}M.\ Razar, {\em  A relation between the two-component of the Tate-Shafarevich group and $L(1)$ for certain elliptic curves},    Amer.\ J.\ Math.\ 96 (1974), 127--144.


\bibitem{rubin87} K.\ Rubin, {\em Tate-Shafarevich groups and $L$-functions of elliptic curves with complex multiplication},   Invent.\ Math. 89 (1987), no.\ 3, 527--559.


\bibitem{rubin91} K.\ Rubin, {\em The ``main conjectures" of Iwasawa theory for imaginary quadratic fields},  Invent.\ Math.\ 103 (1991), no.\ 1, 25--68.


\bibitem{satge}P.\ Satg\'{e}, {\em Un analogue du calcul de Heegner},  Invent.\ Math.\ 87 (1987), no.\ 2, 425--439.

\bibitem{satge2}P.\ Satg\'{e}, {\em  Groupes de Selmer et corps cubiques},   J.\ Number Theory 23 (1986), no.\ 3, 294--317.



\bibitem{sil}J.\ Silverman, {\em  The arithmetic of elliptic curves. Second edition}, Graduate Texts in Mathematics, 106. Springer, Dordrecht, 2009. xx+513 pp.\ ISBN 978-0-387-09493-9.

\bibitem{silverman}J.\ Silverman, {\em   Advanced topics in the arithmetic of elliptic curves},  Graduate Texts in Mathematics, 151. Springer-Verlag, New York, 1994. xiv+525 pp.\ ISBN 0-387-94328-5.



\bibitem{stephens}N.\ Stephens, {\em The diophantine equation $X^3+Y^3=DZ^3$ and the conjectures of Birch and Swinnerton-Dyer},    J.\ Reine Angew.\ Math.\ 231 (1968), 121--162.



\bibitem{zk87} D.\ Zagier, G.\ Kramarz,   {\em Numerical investigations related to the $L$-series of certain elliptic curves},   J.\ Indian Math.\ Soc.\ (N.S.) 52 (1987), 51--69 (1988).



\bibitem{zhao}C.\ Zhao, {\em A criterion for elliptic curves with second lowest 2-power in $L(1)$},   Math.\ Proc.\ Cambridge Philos.\ Soc.\ 131 (2001), no.\ 3, 385--404.


\bibitem{szhang}S.\ Zhang, {\em   The nonvanishing of $L(1)$ for the $L$-series of some elliptic curves},   Adv.\ in Math.\ (China) 24 (1995), no.\ 5, 439--443.

\end{thebibliography}
\end{document}